\documentclass[a4paper,10pt]{amsart}
\usepackage[english]{babel}
\usepackage{amssymb,amsmath,amsthm, mathabx, mathrsfs, enumitem,comment}
\usepackage{thmtools}
\usepackage[margin=1.3in]{geometry}
\usepackage{color}
\definecolor{pink}{rgb}{1,0,1}

\def\beq{\begin{eqnarray}}
\def\eeq{\end{eqnarray}}
\newcommand{\nn}{\nonumber}

\usepackage{booktabs} 
\usepackage{comment} 
 
\usepackage[colorlinks=true, linkcolor=blue,  citecolor=red]{hyperref}
\usepackage[capitalise,nameinlink]{cleveref}

\theoremstyle{plain} 
 
\def\XXint#1#2#3{{\setbox0=\hbox{$#1{#2#3}{\int}$} 
  \vcenter{\hbox{$#2#3$}}\kern-.5\wd0}} 

\newtheorem{thm}{Theorem}[section] 
\newtheorem{cor}[thm]{Corollary} 
\newtheorem{lem}[thm]{Lemma} 
\newtheorem{prop}[thm]{Proposition} 
\newtheorem{conjecture}[thm]{Conjecture} 
\theoremstyle{definition}

\newtheorem{remark}[thm]{Remark}
\newtheorem{claim}[thm]{Claim}
\newtheorem{example}[thm]{Example}
 
 \newcommand{\tTheta}{\tilde{\Theta}}

\newcommand{\Z}{\mathbb{Z}}
\newcommand{\R}{\mathbb{R}}
\newcommand{\C}{\mathbb{C}}
\newcommand{\tr}{\operatorname{Tr}}
\newcommand{\diag}{\operatorname{diag}}

\newcommand{\N}{\mathbb{N}}

\newcommand{\scN}{\mathscr{N}}
\newcommand{\scD}{\mathscr{D}}
\newcommand{\ttheta}{\tilde \theta} 
 
\title{Extrema of spectral determinants on boxes}
\author{Fabio Francesconi}\address{Fabio Francesconi,  Mathematical Sciences, Chalmers University,  412 96, Gothenburg, Sweden}
\author{Julie Rowlett}
\address{Julie Rowlett,  Mathematical Sciences, Chalmers University,  412 96, Gothenburg, Sweden} 
\urladdr{\href{https://www.math.chalmers.se/ rowlett/}{http://www.math.chalmers.se/rowlett}}
\email{\href{mailto:julie.rowlett@chalmers.se}{julie.rowlett@chalmers.se}}
 
\begin{document}

\begin{abstract}
We study the zeta-regularized determinant of the Laplacian on Euclidean boxes of unit volume with three boundary conditions: periodic, Neumann, and Dirichlet. In the periodic case, this amounts to studying the determinant on orthogonal flat tori. We prove that the equilateral box is the unique maximizer in every dimension $n\geq2$ for the periodic and Neumann boundary conditions. In contrast to these cases, the Dirichlet determinant exhibits a strikingly rich, dimension dependent behavior. In dimensions $3$ to $23$, as well as in dimensions $25$, $27$, and $29$, no global maximum exists, whereas a global maximum exists in every dimension $n\ge40$.  We show that this is related to the Casimir energy. We also show that the local behavior near the equilateral box depends on the dimension in a non periodic way. In even dimensions it is always a strict local maximum. In odd dimensions, there is a critical dimension below which it is a strict local minimum and above which it is a strict local maximum. Finally, we study the limiting behavior of the determinant as the dimension tends to infinity and prove that in all three cases the determinant tends to zero as the dimension tends to infinity.
\end{abstract}
\maketitle

\section{Introduction} \label{s:intro}

On a compact Riemannian manifold, the Laplace operator has a discrete spectrum
\[
0=\lambda_0<\lambda_1\leq\lambda_2\leq\cdots\leq\lambda_k\leq\cdots\to\infty.
\]
The associated spectral zeta function, introduced by Minakshisundaram and Pleijel \cite{MinakshisundaramPleijel1949}, is defined analogously to the Riemann zeta function, by replacing the positive integers with the positive eigenvalues of the Laplacian:
\[
\zeta(s)=\sum_{k\geq1}\lambda_k^{-s}.
\]
By Weyl's law \cite{Weyl1912}, the spectral zeta function of a closed $n$-dimensional Riemannian manifold converges for $\Re(s)>n/2$. Formally, if the product below converged and the spectral zeta function were differentiable at $s=0$, one would have
\[
\prod_{k\geq1}\lambda_k=e^{-\zeta'(0)}.
\]
While Weyl's law implies that the product on the left diverges, it also shows that the heat trace converges for every positive time. The spectral zeta function can then be related to the heat trace by
\begin{equation}
\zeta(s)=\frac{1}{\Gamma(s)}
\int_0^\infty t^{s-1}
\tr\left(e^{-t\Delta}-\Pi_{\ker(\Delta)}\right) dt,
\qquad \Re(s)>\frac n2.
\label{eq:zeta_heat}
\end{equation}
Minakshisundaram and Pleijel \cite{MinakshisundaramPleijel1949} used this representation together with the short time asymptotic expansion of the heat kernel to prove that the spectral zeta function admits a meromorphic extension to the complex plane that is holomorphic in a neighborhood of $s=0$. The first coefficients of this expansion were later computed in terms of curvature by McKean and Singer \cite{McKeanSinger1967}, while Seeley \cite{Seeley1967} obtained the meromorphic extension for elliptic operators by a different method, based on their complex powers. These results led Ray and Singer \cite{ray1970,raysinger} to define the zeta-regularized determinant of the Laplacian by

\[
\det\Delta =e^{-\zeta'(0)}.
\]

Ray and Singer were motivated by the Reidemeister--Franz torsion, a topological invariant built from determinants of combinatorial boundary operators. They defined its analytic counterpart using determinants of Laplacians on differential forms, showed that it is independent of the metric, and conjectured that the two torsions agree. This was proved independently by Cheeger \cite{cheeger1,cheegertorsion} and M\"uller \cite{mullertorsion}. Around the same time, Hawking \cite{hawking} introduced zeta regularization in quantum field theory on curved backgrounds, where $e^{-\zeta'(0)}$ regularizes Gaussian path integrals. The technique has since become a standard tool in physics \cite{10useszeta,kkbook}.

The determinant is closely tied to the geometry on which it is defined. One of the most natural questions is: \emph{which geometries maximize or minimize the determinant?} Sarnak \cite{sarnak1}, building on work of Vign\'eras \cite{vigneras}, related determinants of Laplacians on compact Riemann surfaces to Selberg zeta functions. Osgood, Phillips, and Sarnak \cite{ops1} initiated the systematic study of the extremal properties of the determinant, leading to compactness results for isospectral sets of surfaces and planar domains \cite{ops2,ops3}. A key ingredient in this theory is the monotonicity of the determinant along the Ricci flow; see also \cite{kkmonotone,aarmonotone}. In higher dimensions, the picture changes: Okikiolu \cite{okikiolu} showed that in odd dimensions the nature of critical metrics depends on the dimension modulo $4$, with no local maxima when $n\equiv1\pmod4$ and no local minima when $n\equiv3\pmod4$. We will encounter a dimension dependent phenomenon of a different kind for Euclidean boxes with the Dirichlet Laplacian.

A Euclidean box domain is a Cartesian product of intervals
\begin{equation}
\mathcal B=(0,q_1)\times\cdots\times(0,q_n),
\qquad q_j>0,\quad j=1,\ldots,n.
\label{eq:defn_boxdom}
\end{equation}
We consider the non negative Euclidean Laplace operator
\[
\Delta=-\sum_{j=1}^n\frac{\partial^2}{\partial x_j^2}
\]
with three boundary conditions: periodic, Neumann, and Dirichlet. The periodic case is equivalent to studying the Laplacian on an orthogonal flat torus, a closed Riemannian manifold obtained as the quotient of $\mathbb R^n$ by an orthogonal lattice. An orthogonal lattice is a set of points in $\mathbb R^n$ of the form $Q\mathbb Z^n$, where $Q$ is an invertible diagonal matrix.

Consider a one dimensional Euclidean box, which is simply a bounded interval. The positive eigenvalues of the Laplacian with either Neumann or Dirichlet boundary conditions are, for an interval of length $\ell$
\begin{equation}
\frac{k^2\pi^2}{\ell^2},
\qquad k\in\N_+. 
\label{eq:dspec}
\end{equation}
The only difference between the Neumann and Dirichlet spectra is that zero is an eigenvalue in the Neumann case, but not in the Dirichlet case. Consequently, the two spectral zeta functions coincide and are equal to
\[
\frac{\ell^{2s}}{\pi^{2s}}\zeta_R(2s),
\]
where $\zeta_R$ denotes the Riemann zeta function. If instead we consider the periodic boundary condition, equivalently we identify the endpoints of the interval to obtain a one dimensional flat torus, the positive eigenvalues are
\[
4\pi^2\frac{k^2}{\ell^2},
\qquad k\in\N_+,
\]
each with multiplicity two. The corresponding spectral zeta function is therefore
\[
\frac{\ell^{2s}}{(2\pi)^{2s}} 2\zeta_R(2s).
\]
In higher dimensions, the relationship between spectral zeta functions corresponding to these three boundary conditions becomes more complicated, with striking consequences for the extremal properties of the determinant.

Although Euclidean boxes provide one of the most classical settings in which to study the Laplace spectrum \cite{Euler1766,Fourier1822,Lame1833,Lame1861,Rayleigh1877,Rayleigh1900,Jeans1905,CourantHilbert1924}, there remain open questions concerning the extremal properties of their spectral determinants. In dimension two, Faulhuber \cite{faulhuber} studied the zeta-regularized determinant on orthogonal flat tori of unit area, while Rowlett and Aldana \cite{juliealdana} studied the determinant of the Dirichlet Laplacian on rectangles of unit area. In both cases, the unique maximum is attained by the most symmetric geometry: the square torus and the square box, respectively. Moreover, in both settings the determinant tends to zero when the torus or box degenerates. The two-dimensional analysis ultimately relies on the Dedekind eta function, or equivalently, on an associated one dimensional theta function. 

For general flat tori, the hexagonal torus maximizes the determinant in dimension two \cite{ops1}. In higher dimensions, Sarnak and Str\"ombergsson \cite{sarnakstrom} studied the height $\zeta'(0)$, equivalently $-\log\det\Delta$, and showed that the $D_4$, $E_8$, and Leech lattices give local minima of the height and therewith local maxima of the determinant. For the sake of clarity, we refer to the height $\zeta'(0)$ as the \em spectral height. \em Sarnak conjectured that the spectral height of unit-volume flat tori is minimized by the lattice with the longest minimal vector \cite[Conj. 4.5]{chiu}. Combined with Riemann's integral representation of the Epstein zeta function and Montgomery's results on theta functions \cite{Montgomery1988}, these ideas yield a complete answer to the corresponding maximization problem for orthogonal flat tori.

\begin{cor}[Corollary to Theorem 2 in \cite{Montgomery1988}]
\label{th:awesome_torus}
For each $n\geq2$, among all orthogonal $n$-dimensional tori of unit volume, the zeta-regularized determinant of the Laplacian is uniquely maximized by the torus $\mathbb R^n/\mathbb Z^n$.
\end{cor}
The Neumann determinant on Euclidean boxes exhibits a similar behavior.
\begin{thm}
\label{th:awesome_neumann}
For each $n\geq2$, the zeta-regularized determinant of the Neumann Laplacian on box domains of unit volume is uniquely maximized by the equilateral box.
\end{thm}
Moreover, if a family of boxes of unit volume degenerates, with some sides tending to zero and others to infinity, then the Neumann determinant tends to zero, as does the determinant of the associated orthogonal flat torus.  
Our most novel results concern the Dirichlet case, where the behavior turns out to be considerably more subtle. Unlike orthogonal flat tori and Neumann boxes, the existence of global extrema depends on the dimension. In dimension two, the determinant has a unique global maximum at the square box \cite{juliealdana}. In intermediate dimensions, however, the determinant can be unbounded above. In sufficiently high dimensions, a global maximum exists again.

\begin{thm}[Non existence of a global maximum in intermediate dimensions]
\label{thm:no_max_low}
For every
\[
n\in\{3,4,\ldots,23\}\cup\{25,27,29\},
\]
the zeta-regularized determinant of the Dirichlet Laplacian is unbounded from above on boxes of unit volume. In particular, it has no global maximum.
\end{thm}

\begin{thm}[Existence of a global maximum in high dimensions]
\label{thm_global_max}
For each $n\geq40$, the zeta-regularized determinant of the Dirichlet Laplacian attains a global maximum on boxes of unit volume.
\end{thm}

The local behavior near the equilateral box also depends on the dimension. In even dimensions, the equilateral box is always a strict local maximum.  In odd dimensions $3$ to $23$ the equilateral box is a strict local minimum, whereas in odd dimensions $25$ and higher, it is a strict local maximum. 

\begin{thm}[Local extremum]
\label{th:local_extremum}
For each $n\geq2$ even, the zeta-regularized determinant of the Dirichlet Laplacian on box domains of unit volume has a strict local maximum at the equilateral box.  For each $3 \leq n \leq 23$ odd, the zeta-regularized determinant of the Dirichlet Laplacian on box domains of unit volume has a strict local minimum at the equilateral box. For each $n \geq 25$ odd, the zeta-regularized determinant of the Dirichlet Laplacian on box domains of unit volume has a strict local maximum at the equilateral box.
\end{thm}

This suggests that in high dimension the equilateral box is the global maximizer; see Conjecture \ref{conj:dirichlet}.  In our last result, we analyze the determinant in all three settings as the dimension tends to infinity.

\begin{thm}
    \label{thm:dim_infty} 
    The zeta-regularized determinant of the Laplacian on an $n$-dimensional Euclidean box domain with unit volume and with any of the three boundary conditions (periodic, Neumann, or Dirichlet) tends to zero as $n \to \infty$. 
\end{thm}

\subsection{Structure of the paper}
In Section \ref{s:2}, we study the spectral zeta function of orthogonal tori and Neumann boxes and establish fundamental results that will also be used in the Dirichlet case.  Section \ref{s:dirichlet} is devoted to the Dirichlet case. In \S \ref{s:global} we prove Theorems  \ref{thm:no_max_low} and \ref{thm_global_max}, with the former based on a relationship between the determinant and the Casimir energy. In \S \ref{s:local} we compute the Hessian of the spectral height at the equilateral box and prove Theorem \ref{th:local_extremum}; the interval arithmetic computation used in the proof is described in \S \ref{s:numerical}. In \S \ref{s:dimension} we prove Theorem \ref{thm:dim_infty} which shows that in all three cases the determinant tends to zero as the dimension tends to infinity. In \S \ref{s:conclusion} we discuss the implications of our results for the Casimir energy of boxes and physical interpretations. Appendix \ref{app:auxiliary} contains the combinatorial proofs of Proposition \ref{prop:box} and Corollary \ref{cor:box_unit}, and Appendix \ref{app:no_montgomery} gives a proof of Corollary \ref{th:awesome_torus} that does not rely on \cite{Montgomery1988}.

\subsection*{Acknowledgements} 
The authors are grateful to Dennis Eriksson for posing the question which motivated our investigation in \S \ref{s:dimension}, and to an anonymous colleague who drew our attention to  \cite{Montgomery1988}.
\subsection*{AI disclosure} AI tools were used to perform several helpful numerical computations. They were also used to assist in deriving some of the explicit estimates presented in the proof of Theorem \ref{thm:no_max_low}. The authors take full responsibility for the correctness of all results and proofs presented in this paper.

\section{Orthogonal tori and Neumann boxes}

 \label{s:2} 
In the special case of an $n$-dimensional orthogonal torus, the eigenvalues of the Laplacian are explicitly given by
\[
\lambda_{k} = 4\pi^2 \sum_{j=1}^n \frac{k_j^2}{q_j^2}, 
\qquad k = (k_1,\dots,k_n) \in \mathbb{Z}^n.
\]
Here, the torus is obtained as a quotient $\R^n/\diag(q_1, \ldots, q_n)\Z^n$. The eigenvalues in this case are equal to those of the Laplace operator with periodic boundary conditions imposed on the Euclidean box domain $\prod_{j=1} ^n [0, q_j]$.
To express the eigenvalues in a more convenient form, we define the matrix $Q=\diag(q_1, \ldots, q_n)$ and its inverse $Q^{-1}$.   Then, up to the factor of $4\pi^2$,  the Laplace eigenvalues are given by the values of the quadratic form associated with $Q^{-2}$ over $\mathbb{Z}^n$. The  torus spectral zeta function $\zeta_\circ(s)$ can therefore be expressed in terms of the Epstein zeta function:
\beq \zeta_E(s,Q)  = \sum_{k \in \Z^n} ' (Q[k])^{-s}, \quad \zeta_\circ (s) = (2\pi)^{-2s} \zeta_E (s,Q^{-2}). \label{eq:epsteinz} \eeq 
Here 
\[Q[k]= k^{T} Q k, \qquad k\in\mathbb{Z}^n,\]
and the prime on the sum indicates that we are excluding $0 \in \Z^n$ from the sum.  The advantage of this formulation is that the Epstein zeta function has been extensively studied; a non exhaustive list of references is  \cite{Epstein}, \cite{terras}, \cite{sarnakstrom}, \cite{Terras2}.  
One can show by various means including Poisson summation that \eqref{eq:epsteinz} converges for $\Re (s)>n/2$; see \cite{sarnakstrom}.  Moreover, $\zeta_E(s,Q)$ has an analytic continuation to $\mathbb{C}$ except for a simple pole at $s=n/2$ with residue
\[\pi^{n/2}\Gamma\left(\frac{n}{2}\right)^{-1}, \]
when the absolute value of the determinant of $Q$, denoted $|Q|$, is equal to one.  When $Q$ is also diagonal, we have the functional equation,
\[\zeta_E(s,Q^{-2})=\pi^{2s-n/2}\frac{\Gamma(n/2-s)}{\Gamma(s)}\zeta_E(n/2-s,Q^2).\]
We compute the derivative of the spectral zeta function for the orthogonal torus 
\[
\zeta'_\circ(s,Q)
= -2 \log(2\pi) (2\pi)^{-2s} \zeta_E(s,Q^{-2})
+ (2\pi)^{-2s} \zeta'_E(s,Q^{-2}).
\]
Taking the limit as $s \to 0$ and using the functional equation as well as the residue at $n/2$, one finds
\beq 
\zeta_E(0, Q^{-2}) = -1, \label{eq:zetaE(0)}
\eeq 
which is independent of $Q$. 

Therefore,
\beq
\zeta'_\circ(0,Q)
= 2 \log(2\pi) + \zeta'_E(0, Q^{-2}).
\label{eq:zetacirc_zetae} \eeq

This shows that the determinant of the Laplacian on orthogonal flat tori can be viewed as a function on the space of lattices via the associated quadratic form $Q^{-2}$, and that the corresponding variational problem reduces to studying the extrema of $\zeta'_E(0, Q^{-2})$. With this aim, the following well-known result due to Riemann will be useful.  The result can be obtained from the analytic continuation of the zeta function and properties of theta functions; see \cite{terras}.  
\begin{lem}[Riemann]
    The Epstein zeta function $\zeta_E(s,Q)$ may be expressed as
    \begin{align}
        \zeta_E(s,Q)=&\frac{\pi^s}{\Gamma(s)}\left(\frac{|Q|^{-1/2}}{s-n/2}-\frac{1}{s}\right) \nonumber\\
        & + \frac{\pi^s}{\Gamma(s)} 
         \left\{\sum_{a\in\mathbb{Z}^n}'\int_1^\infty \left( e^{-\pi Q[a]t}t^{s-1}+|Q|^{-1/2}e^{-\pi Q^{-1}[a]t} t^{n/2-s-1}\ \right) dt\right\}.\end{align} 
   Here $|Q| = |\det(Q)|$.  Now, assume $|Q| = 1$. Then  
    \begin{align}
        \zeta_E(s,Q^{-2})
         &= \frac{\pi^s}{\Gamma(s)} \left( \frac{1}{s - n/2}    - \frac{1}{s} \right) \nonumber \\
        & + \frac{\pi^s}{\Gamma(s)} 
         \left\{ \sum_{a \in \mathbb{Z}^n }' 
       \int_{1}^{\infty} \left( e^{-\pi Q^{-2}[a]t}   t^{s-1}  
    +  e^{-\pi Q^2[a]t}   t^{ n/2 - s-1}\ \right) dt \right\}.
    \end{align}
    \label{riemann}
\end{lem}
 We take the derivative $\frac{d}{ds}|_{s=0}\zeta_E(s,Q^{-2})$ using the integral expression. Observe that the first term is analytic and independent of $Q$, therefore it contributes only a constant:
 \[\frac{d}{ds}\frac{\pi^s}{\Gamma(s)} \left( \frac{1}{s - n/2}    - \frac{1}{s} \right)\bigg|_{s=0}=C.\]
We set
\[
\mathcal{F}_Q(s)  = \sum_{a \in \mathbb{Z}^n }' 
       \int_{1}^{\infty}\left( e^{-\pi Q^{-2}[a]t}   t^{s-1}  
    +  e^{-\pi Q^{2}[a]t}   t^{ n/2 - s-1}\ \right) dt.\]
Thus,
\begin{equation}
\left. \frac{d}{ds} \zeta_E(s,Q^{-2}) \right|_{s=0}
= C + \left[
\frac{\pi^s}{\Gamma(s)} \mathcal{F}'_Q(s)
+ \frac{\pi^s \log \pi}{\Gamma(s)} \mathcal{F}_Q(s)
+ \frac{\pi^s}{\Gamma(s)} \left( \frac{-\Gamma'(s)}{\Gamma(s)} \right) \mathcal{F}_Q(s)
\right]\bigg|_{s=0}.
\label{eq6}
\end{equation}
At $s=0$, both $\mathcal F_Q(0)$ and $\mathcal F'_Q(0)$ converge due to exponential decay, while $1/\Gamma(0)=0$, so the first two terms vanish.  
Finally, the factor
\[
\frac{\pi^s}{\Gamma(s)} \left( \frac{-\Gamma'(s)}{\Gamma(s)} \right)
\]
is equal to one when evaluated at $s=0$. Hence, in order to study the extrema of \eqref{eq:zetacirc_zetae} it is sufficient to analyze the behavior of 
\begin{equation}
\mathcal{F}(Q) =\mathcal{F}_Q(0) = \sum_{a \in \mathbb{Z}^n}'  \int_1^\infty \left( e^{-\pi Q^{-2}[a]t}t^{-1}
+ e^{-\pi Q^2[a]t}  t^{ n/2 - 1}\ \right) dt.
\label{eq26}
\end{equation}

Chiu proved in \cite{chiu}, Theorem 3.1, that the spectral height of $n$-dimensional flat tori of volume $1$ attains a minimum. In our setting, we restrict to the subspace of orthogonal flat tori with unit volume.  These may be parametrized by the space 
\[\mathcal S_n = \{ q\in(0,\infty)^n : \textstyle\prod_{j=1}^n q_j = 1 \}\]
This is a closed, non compact subset of $\R^n$.  We first show that if unit volume tori collapse, the spectral height tends to infinity and therewith the determinant tends to zero. 

\begin{prop} \label{prop:torus_collapse} The spectral height of a unit-volume orthogonal torus tends to infinity if any side length(s) tend to zero.  
\end{prop}
\begin{proof} 
 Using \eqref{eq26}, to show that the spectral height tends to infinity, it suffices to demonstrate that $\mathcal F(Q)$ tends to infinity.  We therefore estimate for $n \geq 2$, 
\[ \mathcal F(Q) \geq \int_1 ^\infty e^{-\pi q^2 t} dt,\]
with $q$ the smallest diagonal element of $Q$.  Using a change of variables $u = \pi q^2 t$, this gives 
\[\mathcal F(Q) \geq \int_1 ^\infty e^{-\pi q^2 t} dt = \frac{1}{\pi q^2} \int_{\pi q^2} ^\infty e^{-u} du = \frac{e^{-\pi q^2}}{\pi q^2} \to \infty \textrm{ as } q \to 0.\] 
\end{proof} 

As a consequence, there is at least one orthogonal torus of unit volume in each dimension that minimizes the spectral height.  We give an independent proof of Corollary \ref{th:awesome_torus} in Appendix \ref{app:no_montgomery}.  Next we relate the Neumann and Dirichlet spectral heights to the spectral heights of orthogonal tori.

\begin{prop} \label{prop:box}
Let $q_1,\dots,q_n$ be positive real numbers. Let $ Q = \mathrm{diag}(q_1,\dots,q_n) $ and  $\zeta_{\scD}(s,Q), $ $\zeta_{\scN}(s,Q)$ denote the spectral zeta functions of the Laplace operator on the corresponding box domain
\[
\mathcal{B} = (0,q_1)\times \dots \times (0,q_n),\]
with Dirichlet and Neumann boundary conditions,  respectively.  For every subset $J \subseteq \{1,\dots,n\}$ with $|J| = m $, let $ Q_J $ be the  $ m \times m $ diagonal sub-matrix obtained by restricting $ Q $ to the indices in $ J $, and let $Q_J^{-1}$ be its inverse matrix.  Let $\zeta_\circ (s, Q_J)$ denote the spectral zeta function of the Laplace operator for the orthogonal torus associated with the indices in $J$, namely $\R^m/Q_J\Z^m$.
With Dirichlet boundary condition, the following identities hold:
\begin{align*}
\zeta_{\scD}(s,Q)
&=\pi^{-2s}\left[2^{-n}\sum_{m=0}^{n}(-1)^{n-m}\sum_{\substack{J\subseteq\{1,\dots,n\}\\|J|=m}}\zeta_E(s,Q^{-2}_J)\right]\\
&=4^{s}\left[2^{-n}\sum_{m=0}^{n}(-1)^{n-m}\sum_{\substack{J\subseteq\{1,\dots,n\}\\|J|=m}}\zeta_\circ(s,Q_J)\right],
\end{align*}
\[
\zeta_{\scD}(0,Q)=\left(-\frac12\right)^n,
\]
and
\begin{equation}\label{eq5}
\zeta'_{\scD}(0,Q)=-2\log\pi\left(-\frac12\right)^n+2^{-n}\sum_{m=0}^{n}(-1)^{n-m}\sum_{\substack{J\subseteq\{1,\dots,n\}\\|J|=m}}\zeta'_E(0,Q_J^{-2}).
\end{equation}
With Neumann boundary condition, the following identities hold:
\[
\zeta_{\scN}(s,Q)=\pi^{-2s}\left[2^{-n}\sum_{K\subseteq\{1,\dots,n\}}\zeta_E(s,Q_K^{-2})\right]
=4^{s}\left[2^{-n}\sum_{K\subseteq\{1,\dots,n\}}\zeta_\circ(s,Q_K)\right],
\]
\[
\zeta_\scN(0,Q)=\left(\frac12\right)^n-1,
\]
and
\begin{equation}\label{eq:neumann0}
\zeta'_{\scN}(0,Q)=-2\log\pi\left[\left(\frac12\right)^n-1\right]+2^{-n}\sum_{K\subseteq\{1,\dots,n\}}\zeta'_E(0,Q_K^{-2}). 
\end{equation}
Here we used the convention $\zeta_E(s,Q_\emptyset)=0$.
\end{prop}
Since the proof is a combinatorial exercise, we have relegated it to Appendix \ref{app:auxiliary}. 

\begin{cor}\label{cor:box_unit}
Let $Q=\diag(q_1,\dots,q_n)$ with $\det Q=1$. For $\emptyset\neq J\subseteq\{1,\dots,n\}$ let
\[
c_J=\Bigl(\prod_{i\in J}q_i\Bigr)^{-\frac{1}{|J|}},\qquad \tilde Q_J=c_J Q_J ,
\]
so that $\tilde Q_J$ is the matrix of the sub-box $\prod_{i\in J}[0,q_i]$ rescaled to volume one. Then
\begin{align}
\zeta'_{\scD}(0,Q)&=-2\log\pi\left(-\frac12\right)^n+2^{-n}\sum_{m=1}^{n}(-1)^{n-m}\sum_{\substack{J\subseteq\{1,\dots,n\}\\|J|=m}}\zeta'_E(0,\tilde Q_J^{-2}),\label{eq:dir_unit}\\
\zeta'_{\scN}(0,Q)&=-2\log\pi\left[\left(\frac12\right)^n-1\right]+2^{-n}\sum_{\emptyset\neq K\subseteq\{1,\dots,n\}}\zeta'_E(0,\tilde Q_K^{-2}).\label{eq:neu_unit}
\end{align}
That is, \eqref{eq5} and \eqref{eq:neumann0} hold with every $Q_J$ replaced by its rescaling $\tilde Q_J$ to volume one.
\end{cor}
The proof of this corollary is similarly combinatorial and is therefore also contained in Appendix \ref{app:auxiliary}. 

\begin{remark} \label{rem:heattrace}
The coefficient of $t^0$ in the short time asymptotic expansion of the heat trace is precisely the value of the spectral zeta function at $s=0$ minus the contribution given by the dimension of the kernel of the Laplacian, as can be shown using \eqref{eq:zeta_heat}.  Consequently, in the Dirichlet case, 
\[
\zeta_\scD(0,Q)=\left(-\frac{1}{2}\right)^n,
\]
which is independent of $Q$ and depends only on the dimension $n$. In particular, it alternates in sign as $n$ varies. In the Neumann case, the 
\[\zeta_{\scN}(0, Q) =\left(\frac{1}{2}\right)^n-1 < 0, \quad  \forall n \geq 1.\] 
\end{remark} 

\begin{example}
    \label{ss:twodim}
In dimension $2$ we have $Q = \operatorname{diag}(q_1,q_2)$, and equation \eqref{eq5} reads
\[
\zeta'_{\scD}(0,Q)
= -\frac{1}{2}\log\pi
+ \frac{1}{4}\left[
-\zeta'_E(0,q_1^{-2})
-\zeta'_E(0,q_2^{-2})
+\zeta'_E(0,Q^{-2})
\right].
\]
In the one dimensional case, the derivative at zero of the Epstein zeta function can be computed explicitly in terms of the Riemann zeta function, yielding 
\begin{align}
\zeta'_E(0,q_1^{-2}) &= -\log(q_1^2) - 2\log(2\pi), \nonumber\\
\zeta'_E(0,q_2^{-2}) &= -\log(q_2^2) - 2\log(2\pi). \nonumber
\end{align}
Imposing the volume constraint $q_1 q_2 = 1$, i.e.\ $q_2 = 1/q_1$, we obtain
\begin{align}
\zeta'_{\scD}(0,Q)
&= -\frac{1}{2}\log\pi
+ \frac{1}{4}\Big[
\log(q_1^2) + 2\log(2\pi)
+\log(q_2^2) + 2\log(2\pi)
+\zeta'_E(0,Q^{-2})
\Big] \nonumber\\
&= -\frac{1}{2}\log\pi
+ \frac{1}{4}\Big[
\log(q_1^2) + \log(q_2^2)
+ 4\log(2\pi)
+\zeta'_E(0,Q^{-2})
\Big] \nonumber\\
&= -\frac{1}{2}\log\pi
+ \frac{1}{4}\Big[
4\log(2\pi)
+\zeta'_E(0,Q^{-2})
\Big]. \nonumber
\end{align}
This shows that, under the fixed-volume condition, the variational problem reduces to the minimization of $\zeta'_E(0,Q^{-2})$. In particular, this establishes the equivalence between the results obtained by Rowlett and Aldana in \cite[Theorem 5]{juliealdana} and by Faulhuber in \cite[Theorem 1.1]{faulhuber}.
\end{example} 

\subsection{Neumann boxes} \label{s:neumann} 
For a box with the Neumann boundary condition, by \eqref{eq:neumann0}, we can relate the spectral height to that of orthogonal tori of lower dimensions using Corollary \ref{cor:box_unit}.

\begin{proof}[Proof of Theorem \ref{th:awesome_neumann}]
By Corollary \ref{cor:box_unit}
\[ \zeta_\scN'(0, Q) = -2\log\pi\left[\left(\frac12\right)^n-1\right]  + 2^{-n} \sum_{J \subseteq \{1, \ldots, n\}}  \zeta_E'(0, \tilde Q_J^{-2}). \] 
By Corollary \ref{th:awesome_torus}, since these Epstein zeta functions correspond to orthogonal tori with volume equal to one, we have the inequalities 
\[
\zeta_E'(0,\tilde Q^{-2}_J)\ge \zeta_E'(0,I_{|J|}).
\]
Here, $I_{|J|}$ is the $|J| \times |J|$ identity matrix. Consequently, 
\[  \zeta_\scN'(0, Q)  \geq -2\log\pi\left[\left(\frac12\right)^n-1\right] + 2^{-n} \sum_{J \subseteq \{1, \ldots, n\}}\zeta_E'(0, I_{|J|}) = \zeta_\scN'(0, I). \] 
On the right side is the spectral height of the equilateral box with side length one, with respect to the Neumann boundary condition. Consequently, the determinant is maximized by this box.  For $Q\neq I$ the term with $J=\{1,...,n\}$ is strictly larger by the uniqueness in Corollary \ref{th:awesome_torus}.

Finally, we show that if Neumann boxes collapse, the spectral height tends to infinity. Each term $\zeta_E'(0, \tilde Q_J ^{-2}) = C_J + \mathcal F(\tilde Q_J)$, with each $C_J$ depending only on the cardinality of $J$, and each $\mathcal F(\tilde Q_J)$ positive.  For at least one of the $\tilde Q_J$ there is an entry $q$ tending to zero.  Then, we have 
for $|J| \geq 2$, 
\[ \mathcal F(\tilde Q_J) \geq \int_1 ^\infty e^{-\pi q^2 t} dt,\]
with $q$ the smallest diagonal element of $\tilde Q_J$.  
Since the remaining terms are either constants or non negative, this shows that $\zeta_\scN'(0, Q) \to \infty$ as the box collapses.  
\end{proof}

\section{Dirichlet boxes} \label{s:dirichlet}
In the Dirichlet case, the terms appearing in Proposition \ref{prop:box} corresponding to lower dimensional sub-boxes have different signs, and so they compete with each other in the sum.  Therefore, to study the Dirichlet case, our starting point is not the inclusion exclusion principle that gives equation \eqref{eq5}. Consider instead the following splitting. Since the spectrum is invariant under permutations of the coordinates, we can write
\[
\mathcal{B}=\prod_{i=1}^n[0,q_i]=\prod_{i=1}^{n-1}[0,q_i]\times [0,L] = L^{-\frac{1}{n-1}}\mathcal{B}_0 \times[0,L].
\]
Here we set $q_n=L$, and $\mathcal{B}_0=L^{\frac{1}{n-1}}\prod_{i=1}^{n-1}[0,q_i]$ is the $(n-1)$-dimensional box obtained by rescaling each of the intervals in the product by the factor $L^{\frac{1}{n-1}}$ so that the resulting box is similar (it has the same shape but now has unit volume). 
For $d\ge1$, $\mathcal{S}_d$ denotes the set of $d$-dimensional boxes of volume $1$, each box being identified with the ordered $d$-tuple of its side lengths, and we define
\[
\Xi:\ \mathcal{S}_{n-1} \times (0, \infty) \rightarrow \mathcal{S}_n, \qquad (\mathcal{B}_0, L)\mapsto\Xi(\mathcal{B}_0, L)= L^{-\frac{1}{n-1}}\mathcal{B}_0\times [0,L],
\]
where, as above, the side $L$ is labelled as the last one. The map $\Xi$ is a bijection, with inverse
\[
\Xi^{-1}(\mathcal{B})=\Bigl(q_n^{\frac{1}{n-1}}\prod_{i=1}^{n-1}[0,q_i], q_n\Bigr),
\]
thus $(\mathcal{B}_0, L)$ is a well-posed choice of parameters. If boxes are only considered up to isometry, $\Xi$ is still surjective, but it is no longer injective, since any side can play the role of $L$. We denote by $Q_0$ the diagonal matrix of the side lengths of $\mathcal{B}_0$. Thus $\det Q_0=1$, and the section 
\[\prod_{i=1}^{n-1}[0,q_i]=L^{-\frac{1}{n-1}}\mathcal{B}_0 \qquad \text{ corresponds to } \qquad L^{-\frac{1}{n-1}}Q_0=\diag(q_1,\dots,q_{n-1}).\]

The determinant on orthogonal flat tori of unit volume behaves as one would typically expect:  it is maximized by the most uniform geometry, and it tends to zero under degeneration. This reflects the geometry of flat tori:  they are smooth compact Riemannian manifolds without boundary.  In contrast, a box domain has boundary, edges, corners, and a whole combinatorial collection of boundary strata which becomes increasingly complicated as the dimension increases.  These boundary strata are reflected in the heat trace, which in turn appears in \eqref{eq:zeta_heat} for the spectral zeta function.  In the torus case, one has a beautifully symmetric Poisson summation formula: 
\beq \theta(x)  = \sum_{a \in \Z} e^{-\pi a^2 x} = x^{-1/2} \theta(x^{-1}) = x^{-1/2} \sum_{a \in \Z} e^{-\pi a^2/x}, \quad x > 0. \label{eq:poisson_nice} 
\eeq 

In the Dirichlet case, we also have a Poisson summation formula, with a few more terms.  Let 
\beq \tilde \theta(x) = \sum_{a \geq 1} e^{-\pi a^2 x}, \quad x > 0. \label{eq:thetatilde} \eeq 
Then $\tilde \theta(t/L^2)$ is the Dirichlet heat trace of the operator $\frac 1 \pi \Delta$ on an interval of length $L$. We have a Dirichlet-Poisson summation formula 
\beq \tilde \theta(x) &=& \frac 1 2 \theta(x) - \frac 1 2 \nn 
\\ 
&=& \frac 1 2 x^{-1/2} \theta(x^{-1}) - \frac 1 2 \nn \\ 
&=& x^{-1/2} \tilde \theta(x^{-1}) + \frac{1}{2 \sqrt x} - \frac 1 2 \nn \\ 
x \tilde \theta(x^2) &=& \tilde \theta(x^{-2}) - \frac x 2 + \frac 1 2. \label{eq:poisson_dir} \eeq 

From this one can compute the full $t \to 0$ asymptotic expansion of the heat trace of $\frac 1 \pi \Delta$.  Since the heat trace of a product of intervals is the product of the heat traces, this heat trace and the leading terms in its asymptotic expansion for an $n$-dimensional box $\mathcal{B}=\prod_{i=1}^n[0,q_i]$ with $Q=\diag(q_1,\dots,q_n)$ are  
\beq \operatorname{Tr}e^{-\frac{t}{\pi}\Delta} &=& \tilde{\Theta}_Q(t)=\sum_{a\in\N_+^n}e^{-\pi Q^{-2}[a]t}\nn \\ W_Q(t)&=&\sum_{k=0}^nc_kt^{-k/2}=\prod_{i=1}^n\Bigl(\frac{q_i}{2\sqrt t}-\frac12\Bigr), \nn \\ c_0&=&\left(-\frac12\right)^n, \nn \\  E_Q(t)&=&\tilde\Theta_Q(t)-W_Q(t). \label{eq:WQ}
\eeq 
We can also split this as 
\beq \tilde{\Theta}_Q(t) = \tilde \theta(t/L^2) \tilde{\Theta}_{L^{-1/(n-1)}Q_0}(t), \quad Q_0 =L^{\frac{1}{n-1}}\diag(q_1,\dots,q_{n-1}), \quad L = q_n.\label{eq:splitTheta} \eeq 
We then have 
\begin{align*}
\zeta_\scD(s,Q)&=\frac{1}{\Gamma(s)} \int_0 ^\infty \tTheta_Q(\pi t) t^{s-1} dt \\ 
&=\frac{\pi^{-s}}{\Gamma(s)}\int_0^\infty t^{s-1}\Bigl(\frac{L}{2\sqrt{t}}-\frac{1}{2}+\frac{L}{\sqrt{t}}  \tilde \theta(L^2/t) \Bigr)\tilde\Theta_{L^{-1/(n-1)}Q_0}(t)  dt\\
&=\frac{L}{2\sqrt\pi}  \frac{\Gamma\left(s-\frac12\right)}{\Gamma(s)}  \zeta_\scD\bigl(s-\tfrac12,L^{-\frac{1}{n-1}}Q_0\bigr)-\frac{1}{2}  \zeta_\scD\bigl(s,L^{-\frac{1}{n-1}}Q_0\bigr)\\
&\quad+\frac{L\pi^{-s}}{\Gamma(s)}\int_0^\infty t^{s-\frac32} \tilde \theta(L^2/t)  \tilde\Theta_{L^{-1/(n-1)}Q_0}(t)  dt .
\end{align*}
Rescaling a box by $\alpha>0$ multiplies its eigenvalues by $\alpha^{-2}$, so $\zeta_\scD(s,\alpha Q_0)=\alpha^{2s}\zeta_\scD(s,Q_0)$ and
\begin{equation}\label{eq:dir_split}
\begin{split}
\zeta_\scD(s,Q)={}&\frac{L^{\frac{n-2s}{n-1}}}{2\sqrt\pi}  \frac{\Gamma\left(s-\frac12\right)}{\Gamma(s)}  \zeta_\scD\bigl(s-\tfrac12,Q_0\bigr)-\frac{L^{-\frac{2s}{n-1}}}{2}  \zeta_\scD(s,Q_0)\\
&+\frac{L\pi^{-s}}{\Gamma(s)}\int_0^\infty t^{s-\frac32} \ttheta(L^2/t)   \tilde\Theta_{L^{-1/(n-1)}Q_0}(t)  dt .
\end{split}
\end{equation}
The integral in the last term converges for every $s\in\mathbb{C}$, since $\ttheta(L^2/t)$ vanishes exponentially as $t\to0^+$, while as $t\to\infty$ it is $\mathcal{O}(t)$, and $\tilde\Theta_{L^{-1/(n-1)}Q_0}(t)$ decays exponentially. Hence the last term is an entire function of $s$, and \eqref{eq:dir_split} holds for all $s$ by analytic continuation. Observe that the first term in \eqref{eq:dir_split} is regular at $s=0$, and vanishes there, since $\zeta_\scD(\cdot,Q_0)$ is regular at $s=-\frac12$ by Lemma \ref{lem:dir_regular} below, applied in dimension $n-1$. Indeed, $1/\Gamma(s)$ has a simple zero at $s=0$ and $\Gamma(-\frac12)=-2\sqrt{\pi}$.

\begin{lem}[Mellin splitting]\label{lem:dir_regular} 
For every $t_0>0$ and $Q=\diag(q_1,\dots,q_n)$
\[
\zeta'_\scD(0,Q)=c_0\left(\gamma-\log\pi+\log t_0\right)+\int_0^{t_0}E_Q(t)  \frac{dt}{t}+\int_{t_0}^\infty\left(E_Q(t)+c_0\right)\frac{dt}{t}.
\]
Here, $\gamma$ is the Euler-Mascheroni constant. Moreover, $\zeta_\scD(\cdot,Q)$ is regular at $s=0$ and at $s=-\frac12$.
\end{lem}
\begin{proof}
     Consider the heat trace for $\frac 1 \pi \Delta$. For any $n$-dimensional box with $Q=\diag(q_1,\dots,q_n)$, the identity for $\tilde\theta$ \eqref{eq:poisson_dir} applied to each side gives
\[
\tilde\Theta_Q(t)=\prod_{i=1}^{n}\Bigl(\frac{q_i}{2\sqrt{t}}-\frac{1}{2}+\frac{q_i}{\sqrt{t}}  \ttheta(q_i^2/t)\Bigr)=\sum_{k=0}^{n}c_k  t^{-k/2}+E_Q(t),
\qquad c_0=\Bigl(-\frac12\Bigr)^{n},
\]
where $E_Q(t)$ vanishes exponentially as $t\to0^+$. In other words, the small time expansion $\operatorname{Tr}e^{-t\Delta^D}\sim\sum_{k\ge0}a_k t^{(k-n)/2}$ of a box terminates: $a_k=0$ for $k>n$. Splitting the Mellin transform at $t=t_0$, we get for $\Re(s)>n/2$
    \beq 
        \pi^s\Gamma(s)\zeta_\scD(s,Q)
        &=&\int_0^\infty t^{s-1}\tilde\Theta_Q(t)dt \nn \\  &=&\int_0^{t_0}\left(\sum_{k=0}^{n}c_k  t^{-k/2}+E_Q(t)\right)t^{s-1}dt+\int_{t_0}^\infty t^{s-1}\tilde\Theta_Q(t)dt \nn \\ 
        &=& 
     \sum_{k=0}^n\frac{c_k   {t_0}^{s-k/2}}{s-k/2}+\int_0^  {t_0} t^{s-1}E_Q(t)dt+\int_{t_0}^\infty t^{s-1}\tilde\Theta_Q(t)dt.
    \label{eq:mellinsplit1} 
    \eeq 
    The two integrals in \eqref{eq:mellinsplit1} define entire functions of $s$, given the exponential decay of $E_Q(t)$ as $t\to0^+$ and of $\tilde\Theta_Q(t)$ as $t\to\infty$. The sum in \eqref{eq:mellinsplit1} is meromorphic on $\C$, with poles at $s=k/2$ for $k=0, 1, \ldots, n$. Hence, we can meromorphically extend the last expression above to the whole complex plane. In particular, expanding around $s=0$, last expression above equals 
    \beq 
     \frac{c_0}{s} &+& A + \mathcal O (s),  \nn \\ 
    A&=&  c_0\log{t_0}-\sum_{k=1}^n\frac{2c_k}{k}  {t_0}^{-k/2}+\int_0^{t_0}E_Q(t)  \frac{dt}{t}+\int_{t_0}^\infty\tilde\Theta_Q(t)  \frac{dt}{t}.  \label{eq:mero_s_zero}
    \eeq 

    At the same time we have near $s=0$ 
    \[\frac{\pi^{-s}}{\Gamma(s)}=s+(\gamma-\log\pi)s^2+\mathcal{O}(s^3).\] Hence $\zeta_\scD(0,Q)=c_0$ $=(-\frac12)^n$ as one would expect from Proposition \ref{prop:box}, and $\zeta'_\scD(0,Q)=A+c_0(\gamma-\log\pi)$. Finally, 
    \[\sum_{k=1}^n\frac{2c_k}{k}{t_0}^{-k/2}=\int_{t_0}^\infty\bigl(W_Q(t)-c_0\bigr)\frac{dt}{t} \quad\text{ and }\quad \tilde\Theta_Q(t)-W_Q(t)+c_0=E_Q(t)+c_0.\]
Thus $\zeta_\scD(s,Q)$ has poles only at $s=\frac12,1, \frac 3 2, \dots,\frac n2$, and in particular $\zeta_\scD(s,Q)$ is regular at $s=-\frac12$. 
\end{proof}
Taking the derivative of \eqref{eq:dir_split} at $s=0$, and using $1/\Gamma(s)=s+\mathcal{O}(s^2)$, $\zeta_\scD(0,Q_0)=(-\frac12)^{n-1}$, together with the representation of the Bessel function $K_\nu$ with $\nu=1/2$ \cite[10.32]{DLMF}, that gives 
\[\int_0^\infty t^{-3/2}e^{-A/t-Bt}  dt=\sqrt{\pi/A}  e^{-2\sqrt{AB}} \] 
we obtain
\begin{equation}
    \zeta'_\scD(0,Q)=-\zeta_\scD\!\left(-\tfrac12,Q_0\right)L^{\frac{n}{n-1}}
+\frac{(-1/2)^{n-1}}{n-1}  \log L
-\tfrac12\zeta'_\scD(0,Q_0)
+\mathcal{O}\left(e^{-\upsilon L^{n/(n-1)}}\right)
\label{eq:dir_mess}
\end{equation}
as $L\to\infty$, for some $\upsilon>0$ depending only on $\mathcal{B}_0$. More precisely, \eqref{eq:dir_mess} holds for every $L>0$ with the remainder
\[
-\sum_{b\in\N_+^{n-1}}\log\Bigl(1-e^{-2L^{n/(n-1)}\sqrt{\lambda_b}}\Bigr)>0,
\]
where $\lambda_b$, $b\in\N_+^{n-1}$, are the Dirichlet eigenvalues of $\mathcal{B}_0$.

\subsection{Boundary behavior and the existence of global extrema}\label{s:global}
In this subsection we study whether the determinant $\det\Delta^D_{\mathcal B}$ attains a global maximum on $\mathcal S_n$. By \eqref{eq:dir_mess}, for a fixed $\mathcal B_0$ the behavior of the spectral height along the family $\Xi(\mathcal B_0, L)$ as $L\to\infty$ is governed by the coefficient $\zeta_\scD\bigl(-\tfrac12,Q_0\bigr)$. In particular, if 
\[\zeta_\scD\bigl(-\tfrac12,Q_0\bigr)>0 \quad \text{ for some }\quad \mathcal B_0\in\mathcal S_{n-1},\] then the spectral height is not bounded from below on $\mathcal S_n$. Equivalently, the zeta-regularized determinant is not bounded from above on $\mathcal S_n$, and in particular it has no global maximum. The coefficient $\zeta_\scD\bigl(-\tfrac12,Q_0\bigr)$ is a familiar object in physics. Since $\zeta_\scD(s,Q_0)=\sum_k\lambda_k^{-s}$, where $\lambda_k$ are the Dirichlet eigenvalues of $\mathcal B_0$, the zeta-regularized quantity $\frac12\zeta_\scD\bigl(-\tfrac12,Q_0\bigr)$ is the Casimir energy 
of a massless scalar field in $\mathcal B_0$ with Dirichlet boundary conditions (in units $\hbar=c=1$). Thus \eqref{eq:dir_mess} states that, along elongated boxes of dimension $n$ with one side length much longer than the remaining $n-1$ dimensional cross-section, the determinant is governed by the Casimir energy of the cross-section. The sign of this energy for rectangular cavities and hypercubes has been studied in the physics literature; see, for instance, \cite{AmbjornWolfram,ChengLiLiZhai,LiZhai,Edery}. For hypercubes with an odd number of finite edges the energy can be negative \cite{LiZhai}. With unequal edges it can be positive \cite{ChengLiLiZhai}, and the behavior of hypercubes in high dimension has been investigated numerically \cite{Edery}. However, these sign results rely largely on numerical evaluation. Here we give analytic proofs of the sign results we need, in all the dimensions covered by Theorem \ref{thm:no_max_low}. We know from \cite{juliealdana} that in dimension $2$ a global maximum for the zeta regularized determinant on $\mathcal{S}_2$ exists. The following proposition shows that in dimension 3, the determinant is unbounded; we will see that the extrema of the determinant depend on the dimension in a somewhat subtle way.

\begin{prop}
    Let $n=3$. The determinant of the Dirichlet Laplacian on the space of boxes with unit volume is not bounded from above.
\end{prop}
\begin{proof}
    By \eqref{eq:dir_mess} it suffices to find $\mathcal{B}_0\in\mathcal{S}_2$ such that
    \[\zeta_\scD\!\left(-\tfrac12,Q_0\right)>0.\]
    Consider the unit square $\mathcal{B}_0=[0,1]\times[0,1]$, so that $Q_0=I_2$. Then, by Proposition \ref{prop:box},
    \[\zeta_\scD\!\left(-\tfrac12,Q_0\right)=\frac{\pi}{2}\Bigl[\frac{1}{2}  \zeta_E\!\left(-\tfrac12,I_2\right)-\zeta_E\!\left(-\tfrac12,1\right)\Bigr],\]
    where $\zeta_E(s,1)=2\zeta_R(2s)$ and, by Jacobi's two-square theorem, $\zeta_E(s,I_2)=4\zeta_R(s)\beta(s)$ \cite{BGMWZ13}, with $\beta$ the Dirichlet beta function. Thus
    \[\zeta_\scD\!\left(-\tfrac12,Q_0\right)=\pi\bigl[\zeta_R(-\tfrac12)\beta(-\tfrac12)-\zeta_R(-1)\bigr].\]
    The reflection formulas \cite{DLMF} of $\zeta_R$ and $\beta$ give $\zeta_R(-\frac12)=-\frac{\zeta_R(3/2)}{4\pi}$ and $\beta(-\frac12)=\frac{\beta(3/2)}{\pi}$, and $\zeta_R(-1)=-\frac{1}{12}$, so that
    \[\zeta_\scD\!\left(-\tfrac12,Q_0\right)=\frac{\pi}{12}-\frac{\zeta_R(3/2)  \beta(3/2)}{4\pi}.\]
    Then we use the standard estimates
    \[
        \zeta_R(3/2)<1+\int_1^\infty x^{-3/2}  dx=3,\qquad
        \beta(3/2)=1-3^{-3/2}+5^{-3/2}-\dots<1,
    \]
    to get
    \[\zeta_\scD\!\left(-\tfrac12,Q_0\right)>\frac{\pi}{12}-\frac{3}{4\pi}=\frac{\pi^2-9}{12\pi}>0. \qedhere\]
\end{proof}

In contrast with dimension $3$, in high dimension the determinant is bounded from above. Theorem \ref{thm_global_max} states that for every $n\ge40$ the zeta-regularized determinant of the Dirichlet Laplacian attains a global maximum on $\mathcal S_n$. The proof of Theorem \ref{thm_global_max} relies on deriving a lower bound for $\zeta'_{\scD}(0,Q)$ and showing that, when $n\ge40$, it tends to $+\infty$ as the shortest side of the box tends to zero. Consequently, by continuity of the map $q \mapsto \zeta'_\scD(0,Q)$, there exists $\varepsilon >0$ such that the minimum of $\zeta'_\scD(0,Q)$ on $\mathcal S_n$ is contained in
\[ \left\{(q_1, \ldots, q_n) \in (0,\infty)^n : \prod_{j=1} ^n q_j =  1, \quad  \varepsilon \le |q_j| \le \varepsilon^{-1}, \quad   \forall j\right\}. \]
This is now a compact set, so by the continuity of $q \mapsto \zeta'_\scD(0,Q)$, this function achieves at least one minimum within this set. The determinant achieves its maximum precisely when $\zeta'_\scD(0,Q)$ achieves its minimum.

In Lemma \ref{lem:lower} we compute the lower bound of $\zeta_\scD'(0, Q)$, while in Lemma \ref{lem:vertex} we obtain the key estimate on the terms in this lower bound, to guarantee the positivity of the coefficients and the divergence of $\zeta'_\scD(0,Q)$ when $n\ge40$.  In what follows, for $k\ge0$, we set
\beq \label{eq:def_P_m}
\mathcal P^{(k)}_m &=& 2\int_1^\infty v^{k}(v-1)^{m-1}e^{-\pi v^2}dv, \qquad m\ge1, \\ 
\chi_m &=& \begin{cases} 1, & m \textrm{ is even,} \\ 0, & m \textrm{ is odd,} \end{cases} \nn \\ 
\lbrack x \rbrack^+ &=& \max \{ x, 0\}. \nn 
\eeq

\begin{lem}\label{lem:lower}
Let $Q$ be the matrix associated to a box $\mathcal{B}\in \mathcal S_n$ and $q=q_{i_0}$ be its shortest side, and let $y_i=q/q_i\in(0,1]$, so that $y_{i_0}=1$. We set $y=(y_1,\dots,y_n)$. Then
\begin{equation}
    \label{eq:zeta_box_est}
\zeta'_\scD(0,Q)\ \ge\ c_0\bigl(\gamma-\log\pi+2\log q\bigr)+2^{1-n}q^{-n}\bigl(\mathcal G_+(y)-\mathcal G_-(y)\bigr),
\end{equation}
where
\[
\mathcal G_+(y)=2\int_1^\infty e^{-\pi u^2}\prod_{i\ne i_0}(u-y_i)  du,\qquad
\mathcal G_-(y)=\int_0^1\bigl[R_y(u)\bigr]^+du,
\]
and $R_y$ is the polynomial
\[
R_y(u)=\frac1u\Bigl(\prod_{i=1}^n(u-y_i)-\prod_{i=1}^n(-y_i)\Bigr).
\]
\end{lem}
\begin{proof}
    We apply Lemma \ref{lem:dir_regular} with $t_0=q^2$ and we substitute $t=q^2/u^2$ so that $\frac{dt}{t}=-\frac{2du}{u}$: 
    \begin{align}\label{eq:estimate_polynomial}
        \zeta'_\scD(0,Q)=&c_0\big(\gamma-\log\pi+2\log q\big)+\int_0^{q^2}E_Q (t) \frac{dt}{t}+\int_{q^2}^{\infty}(E_Q(t)+c_0)  \frac{dt}{t}=\nonumber\\
        &c_0\big(\gamma-\log\pi+2\log q\big)+2\int_1^\infty E_Q\Big(\frac{q^2}{u^2}\Big)\frac{du}{u}+2\int_0^1 \left(E_Q\Big(\frac{q^2}{u^2}\Big)+c_0\right)\frac{du}{u}.
    \end{align}
    We write $\tilde\theta(t/q_i^2)=p_i+r_i,$ with
    \[p_i = p_i(t)=\frac{q_i}{2\sqrt{t}}-\frac{1}{2}=\frac{u-y_i}{2y_i}\quad\text{ and }\quad r_i = r_i(t)=\frac{q_i}{\sqrt{t}}\tilde\theta(q_i^2/t)>0.\]
    Since $\prod_i y_i=q^n$, by definition of $W_Q$
    \begin{equation}
        W_Q=\prod_{i=1}^n\Bigl(\frac{q_i}{2\sqrt t}-\frac12\Bigr)=\prod_{i=1}^n p_i=\frac{\prod_i(u-y_i)}{2^n\prod_i y_i}=2^{-n}q^{-n}\prod_{i=1}^n(u-y_i),
    \end{equation}
    \begin{equation}\label{eq_c_0_sub}
        c_0=\Big(-\frac12\Big)^n=2^{-n}(-1)^n=2^{-n}q^{-n}(-1)^n\prod_i y_i=2^{-n}q^{-n}\prod_{i=1}^n(-y_i).
    \end{equation}
    For $t\le q^2$, that is $u\ge1$, all the $p_i$ are nonnegative, hence, expanding the product
    \begin{align*}
        E_Q=\prod_{i=1}^n(p_i+r_i)-\prod_{i=1}^n p_i =& \sum_{\emptyset\ne S\subseteq\{1,\dots,n\}}\ \prod_{i\in S}r_i\prod_{i\notin S}p_i \ge \\ &
        r_{i_0}\prod_{i\ne i_0}p_i=r_{i_0}\frac{\prod_{i\ne i_0}(u-y_i)}{2^{n-1}\prod_{i\ne i_0}y_i}=\\&r_{i_0}2^{1-n}q^{-n}\prod_{i\ne i_0}(u-y_i)\ge\\& 2^{1-n}q^{-n}ue^{-\pi u^2}\prod_{i\ne i_0}(u-y_i).
    \end{align*}
    Here, in the first inequality we restricted to the set $S=\{i_0\}$. In the second we firstly applied the change  of variable $\sqrt{t}=q/u$ in the definition of $r_i$ and then, for the shortest side $q_{i_0}$
    \[r_{i_0}=u\sum_{a\ge1}e^{-\pi a^2 u^2}\ge ue^{-\pi u^2}.\] It follows that 
     \begin{equation}\label{eq:I_1_estimate}
         2\int_1^\infty E_Q\Big(\frac{q^2}{u^2}\Big)\frac{du}{u}\ge\ 2^{1-n}q^{-n}\cdot 2\int_1^\infty e^{-\pi u^2}\prod_{i\ne i_0}(u-y_i)du=2^{1-n}q^{-n}\mathcal G_+(y).
     \end{equation}
     Finally, for $t>q^2$ that is $u<1$, since $\tilde\Theta_Q\ge0$
     \[E_Q+c_0=\tilde\Theta_Q-(W_Q-c_0)\ \ge\ -(W_Q-c_0).\]
     Substituting the expression for $c_0$ \eqref{eq_c_0_sub}
     \[W_Q-c_0=2^{-n}q^{-n}\Big(\prod_i(u-y_i)-\prod_i(-y_i)\Big).\]
     the polynomial in the parenthesis vanishes at $u=0$, hence we can write the parenthesis as $uR_y(u)$ and 
     \[W_Q-c_0=2^{-n}q^{-n}uR_y(u) \qquad\text{and}\qquad E_Q+c_0\ \ge\ -2^{-n}q^{-n}  u  [R_y(u)]^+ .\]
     We get 
     \begin{equation}\label{eq:I_2_estimate}
         2\int_0^1 (E_Q+c_0)\Big(\frac{q^2}{u^2}\Big)\frac{du}{u}\ge-2\cdot2^{-n}q^{-n}\int_0^1[R_y(u)]^+du=-2^{1-n}q^{-n}  \mathcal G_-(y).
     \end{equation}
     Finally, substituting \eqref{eq:I_1_estimate} and \eqref{eq:I_2_estimate} into \eqref{eq:estimate_polynomial} 
     completes the proof. 
  
\end{proof}

In the next Lemma we find an estimate for the second term on right hand side of equation \eqref{eq:zeta_box_est}. For this we recall the Euler Beta function 
\beq \mathrm B(a,b) = \int_0 ^1 t^{a-1} (1-t)^{b-1} dt = \frac{\Gamma(a) \Gamma(b)}{\Gamma(a+b)}. \label{eq:eulerbeta} \eeq 

 \begin{lem}\label{lem:vertex}
Let $y\in[0,1]^n$ with $y_{i_0}=1$, and let
\begin{equation}\label{eq:def_delta_n}
    \mu_n=\min_{1\le m\le n-1}\Bigl(\mathcal P^{(n-m)}_m-\chi_m  \mathrm B(n-m,m+1)\Bigr).   
\end{equation}
Then $\mathcal G_+(y)-\mathcal G_-(y)\ge(1-y_j)  \mu_n-n  y_j$ for every $j\ne i_0$. Moreover, for every $y\in[0,1]^n$, $\mathcal{G}_-(y)\le n.$
\end{lem}    
\begin{proof}
    Since we must prove that the estimate holds for every $j \neq i_0$, we begin by fixing such a $j$. Let $y^0$ be the vector $y$ with the $j$-th coordinate $y_j$ replaced by $0$. For $u\ge1$ we have $u-y_j\ge(1-y_j)u$, and all other factors in
    \[\mathcal G_+(y)=2\int_1^\infty e^{-\pi u^2}(u-y_j)\prod_{i\ne i_0,j}(u-y_i)du\]
    are nonnegative, thus
    \begin{equation}\label{eq:vertex_Gplus}
        \mathcal G_+(y)\ge(1-y_j)\mathcal{G}_+(y^0).
    \end{equation}
    Consider $\mathcal G_-$ and $R_{y^0}(u)=\prod_{i\ne j}(u-y_i)$. Then $\prod_i(u-y_i)=(u-y_j)R_{y^0}(u)$ and $\prod_i(-y_i)=(-y_j)\prod_{i\ne j}(-y_i)=-y_j  R_{y^0}(0)$, so that
    \[
    R_y(u)=\frac{(u-y_j)R_{y^0}(u)+y_j  R_{y^0}(0)}{u}=R_{y^0}(u)-y_j  \frac{R_{y^0}(u)-R_{y^0}(0)}{u}.
    \]
    On $[0,1]$, $R_{y^0}$ is a product of $n-1$ factors $w-y_i$ with $|w-y_i|\le1$, hence $|R_{y^0}|\le1$. Moreover
    \[
    R_{y^0}'(w)=\sum_{k\ne j}\ \prod_{i\ne j,k}(w-y_i)
    \]
    is a sum of $n-1$ products of factors of modulus at most $1$, hence $|R_{y^0}'|\le n-1$ on $[0,1]$. By the mean value theorem, $|R_{y^0}(u)-R_{y^0}(0)|\le(n-1) u$, and therefore $R_y(u)\le R_{y^0}(u)+(n-1) y_j$ and $[R_y(u)]^+\le[R_{y^0}(u)]^++(n-1) y_j$ for $0<u\le1$. Integrating over $[0,1]$, we get
    \begin{equation}\label{eq:vertex_Gminus}
    \mathcal G_-(y)\le\mathcal G_-(y^0)+(n-1) y_j,\qquad \mathcal G_-(y^0)\le 1,
    \end{equation}
    where the last bound holds because $0\le[R_{y^0}]^+\le|R_{y^0}|\le1$ on $[0,1]$. Thus $\mathcal G_-(y)\le n$
    Combining \eqref{eq:vertex_Gplus} and \eqref{eq:vertex_Gminus},
    \[
    \begin{aligned}
    \mathcal G_+(y)-\mathcal G_-(y)&\ge(1-y_j)  \mathcal G_+(y^0)-\mathcal G_-(y^0)-(n-1)  y_j\\
    &=(1-y_j)\bigl(\mathcal G_+(y^0)-\mathcal G_-(y^0)\bigr)-y_j  \mathcal G_-(y^0)-(n-1)  y_j\\
    &\ge(1-y_j)\bigl(\mathcal G_+(y^0)-\mathcal G_-(y^0)\bigr)-n  y_j .
    \end{aligned}
\]
Since $1-y_j\ge0$, it remains to show that $\mathcal G_+(y^0)-\mathcal G_-(y^0)\ge\mu_n$. In $y^0$ the coordinates $y_{i_0}=1$ and $y_j=0$ are fixed, while the remaining $n-2$ coordinates $y_i$, $i\notin\{i_0,j\}$, range in the cube $[0,1]^{n-2}$. We consider $\mathcal G=\mathcal G_+-\mathcal G_-$ as a function of these variables. It is continuous on $[0,1]^{n-2}$, since $R_y(u)$ is a polynomial in $(u,y)$. We claim that $\mathcal{G}$ is concave in each variable separately. Indeed,
\begin{enumerate}
    \item[(i).] $\mathcal G_+$ is affine in each $y_i$, since $y_i$ appears in exactly one factor of $\prod_{l\ne i_0}(u-y_l)$. Hence the integrand, and so its integral in $u$, is affine in $y_i$.
    \item[(ii).] $\mathcal G_-$ is convex in each $y_i$. Indeed, both $\prod_l(u-y_l)$ and $\prod_l(-y_l)$ are affine in $y_i$, so the same holds for $R_y(u)$ for every fixed $u$. Since $x\mapsto[x]^+$ is convex, $[R_y(u)]^+$ is convex in $y_i$, and so is its integral in $u$.
\end{enumerate}
Let $y^*$ be a minimum point of $\mathcal G$ on $[0,1]^{n-2}$. Fix all coordinates but $y_i$. Since $\mathcal G$ is concave in $y_i$, its minimum is attained at one of the endpoints $0$ or $1$. We can therefore replace $y_i^*$ by one of these endpoints and obtain another minimum point. Repeating this argument for each coordinate, we obtain a minimum point with all coordinates in $\{0,1\}$, hence a vertex of $[0,1]^{n-2}$ and
\[
\min_{[0,1]^{n-2}} \mathcal G=\min_{\text{vertices}}\mathcal G .
\]
At a vertex every $y_i$ equals $0$ or $1$. Let $m$ be the number of coordinates equal to $1$, including $y_{i_0}$.  Since $y_j=0$, we have $1\le m\le n-1$. Besides $i_0$ there are $m-1$ coordinates equal to $1$ and $n-m$ equal to $0$, so
\[
\prod_{l\ne i_0}(u-y_l)=(u-1)^{m-1}u^{n-m},
\]
and therefore 
\[
\mathcal G_+=2\int_1^\infty u^{n-m}(u-1)^{m-1}e^{-\pi u^2}  du=\mathcal P^{(n-m)}_m.
\]
Since at least one coordinate vanishes, $\prod_l(-y_l)=0$ and
\[
R_{y^0}(u)=\frac{(u-1)^m u^{n-m}}{u}=(-1)^m(1-u)^m u^{n-m-1},\qquad 0<u<1 .
\]
Its positive part is $(1-u)^m u^{n-m-1}$ if $m$ is even and $0$ if $m$ is odd, that is $\chi_m(1-u)^m u^{n-m-1}$. Thus
\[
\mathcal G_-=\chi_m\int_0^1u^{n-m-1}(1-u)^m  du=\chi_m  \mathrm B(n-m,m+1).
\]
Hence at every vertex $\mathcal G=\mathcal P^{(n-m)}_m-\chi_m\mathrm B(n-m,m+1)\ge\mu_n$, so that $\mathcal G_+(y^0)-\mathcal G_-(y^0)\ge\mu_n$, which concludes the proof.
\end{proof}

\begin{proof}[Proof of Theorem \ref{thm_global_max}]
Since $\det Q=1$ we have $q\le1$ and 
\[\prod_{i\neq i_0} y_i= q^{n-1}/(q^{-1})=q^n,\]
so, there is a $j\neq i_0$ with $y_j\le q^{n/(n-1)}\le q$. Thus, by Lemmas \ref{lem:lower} and \ref{lem:vertex} (applied with this $j$), together with $|c_0|=2^{-n}$ and $y_j\le q$, whenever $\mu_n\ge0 $ we have the inequality
\begin{equation}\label{eq:globla_max_thm}
    \zeta'_{\scD}(0,Q)\ge -2^{-n}\left(|\gamma-\log \pi|+2|\log q|\right)+2^{1-n}q^{-n}\left((1-q)\mu_n-nq\right).
\end{equation}
We show that $\mu_n>0$ for $n\ge40$. By definition of $\mathcal{P}$ \eqref{eq:def_P_m}
\[\mathcal{P}^{(n-m)}_m=2\int_1^\infty v^{n-m}\left(v-1\right)^{m-1}e^{-\pi v^2}dv=2\int_1^\infty v^{n-1}\left(1-\frac{1}{v}\right)^{m-1}e^{-\pi v^2}dv.\]
For $v\ge1$ the term in the parenthesis is strictly lower than 1, thus the function in the integrand is a decreasing function of $m$ and so  is $\mathcal{P}_m^{(n-m)}$. Hence
\[\mathcal{P}_m^{(n-m)}\ge \mathcal{P}^{(1)}_{n-1}, \qquad\text{for }m\le n-1.\]
Moreover 
\begin{equation}\label{eq:vertex_bound_1}
    \mathrm B(n-m,m+1)=\frac{(n-m-1)!  m!}{n!}=\left(n\binom{n-1}{m}\right)^{-1}\le\frac1n.
\end{equation}
Next, consider
\begin{equation}\label{eq:vertex_bound_2}
    \begin{aligned}
    \mathcal{P}^{(1)}_{n-1}=&2\int_1^\infty v  (v-1)^{n-2}e^{-\pi v^2}dv\ge\\
    &2^{n-1}\int_3^\infty ve^{-\pi v^2}dv=\frac{2^{n-2}}{\pi e^{9\pi}} .
\end{aligned}
\end{equation}
The inequality follows from the restriction to $v\ge3$ in the integral, where $v-1\ge2$. $\pi<4$ and $e^{\pi}<24$ give
\[\pi e^{9\pi}<4\cdot24^9=2^{29}\cdot 3^9\le2^{29}\cdot 40\cdot2^{9}=40\cdot2^{38},\]
so $n 2^{n-2}>\pi e^{9\pi}$ and $\mu_n\ge\frac{2^{n-2}}{\pi e^{9\pi}}-\frac{1}{n}>0$ for every $n\ge40$. For each fixed value of $n$, the right-hand side of \eqref{eq:globla_max_thm} depends only on $q$ and tends to
$+\infty$ as $q\to0$, from which the theorem follows.
\end{proof}

The lower bound in \eqref{eq:zeta_box_est} follows from the dominant term in the heat trace \[\tilde\Theta_Q=W_Q+E_Q.\] In the proof of Lemma \ref{lem:lower}, we apply Lemma \ref{lem:dir_regular} with $t_0=q^2$. With this choice, the heat coefficients
\[
W_Q-c_0=\sum_{k\ge1}c_kt^{-k/2}
\]
contribute only for $t>q^2$, whereas the contribution for $t<q^2$ comes entirely from the remainder $E_Q$. More precisely, the Mellin transform for times $t>q^2$ gives $-2^{1-n}q^{-n}\mathcal G_-(y)$, while the times $t<q^2$ give $2^{1-n}q^{-n}\mathcal G_+(y)$. Hence, the behavior of the spectral height at the boundary of $\mathcal S_n$ is determined by the sign of $\mathcal G_+(y)-\mathcal G_-(y)$, which measures the competition between the remainder and the heat coefficients. This difference depends only on $n$ and the ratios $y_i$. The dependence on the dimension therefore comes from the difference in how the two terms scale as $n$ increases. Indeed, $\mathcal G_-$ is an integral over $[0,1]$ of a product of factors of modulus at most one, and at the vertices of $[0,1]^{n-2}$ (where $\mathcal G_+(y)-\mathcal G_-(y)$ attains its minimum) equation \eqref{eq:vertex_bound_1} bounds it from above by $1/n$. On the other hand, in $\mathcal G_+$, the exponentially small remainder $ue^{-\pi u^2}$ associated with the shortest side is multiplied by the leading terms $u-y_i$ associated with the other $n-1$ sides. At the vertices, this gives a lower bound of $2^{n-2}/(\pi e^{9\pi})$ in \eqref{eq:vertex_bound_2}. Thus, as the dimension increases, the contribution of the remainder eventually dominates. However, in dimensions $3$ to $22$, as well as the odd dimensions from $23$ to $29$, the leading terms in the heat trace dominate, and they cause the Dirichlet determinant to be unbounded on the space of unit volume Euclidean boxes.

 Indeed, Theorem \ref{thm:no_max_low} states that for every $n\in\{3,4,\dots,23\}\cup\{25,27,29\}$ the determinant of the Dirichlet Laplacian is not bounded from above on $\mathcal S_n$, so that, unlike in dimension $2$ and in every dimension $n\ge40$, it has no global maximum. By \eqref{eq:dir_mess} and the remark following it, the determinant is unbounded from above on $ \mathcal S_n $ whenever there exists a section \( \mathcal B_0 \in\mathcal S_{n-1} \) such that \( \zeta_\scD(-\frac12,Q_0)>0 \). The proof of Theorem \ref{thm:no_max_low} therefore reduces to establishing the existence of such a section. The first step is Lemma \ref{lem:zeta_half} which expresses \(\zeta_\scD(-\frac12,Q)\) more conveniently.

\begin{lem}\label{lem:zeta_half}
For every $n$-dimensional box with matrix $Q=\diag(q_1,\dots,q_n)$,
\[
\zeta_\scD\bigl(-\tfrac12,Q\bigr)=-\frac12\int_0^\infty t^{-3/2}E_Q(t)  dt
=-2^{-n}\int_0^\infty\Bigl(\prod_{i=1}^n\psi(q_iv)-\prod_{i=1}^n(q_iv-1)\Bigr)dv ,
\]
with $E_Q$ defined in \eqref{eq:WQ}, and 
\beq \psi(x) = 2 \tilde \theta (x^{-2}). \label{eq:defpsi} \eeq 
\end{lem}
\begin{proof}
By the proof of Lemma   \ref{lem:dir_regular}, for every $s\in\C$,
\[
\pi^s\Gamma(s)  \zeta_\scD(s,Q)=\sum_{k=0}^n\frac{c_k}{s-k/2}+\int_0^1t^{s-1}E_Q(t)  dt+\int_1^\infty t^{s-1}\tilde\Theta_Q(t)  dt.
\]
If $\Re s<0$, then $\frac{c_k}{s-k/2}=-c_k\int_1^\infty t^{s-1-k/2}dt$ for every $k\ge0$. Hence, since the sum is finite, exchanging the summation and integration and using the definition \eqref{eq:WQ} of $E_Q(t)$ 
\[
\pi^s\Gamma(s)  \zeta_\scD(s,Q)=\int_0^\infty t^{s-1}E_Q(t)  dt\qquad\text{for }\Re s<0.
\]
The integral converges absolutely, because $E_Q$ vanishes exponentially at $t=0$ and is bounded on $[1,\infty)$. Taking $s=-\frac12$ and using $\Gamma(-\frac12)=-2\sqrt\pi$ gives the first identity. For the second, substitute $t=v^{-2}$, so that $t^{-3/2}dt=-2  dv$ and use \eqref{eq:defpsi} to show that 
\[\tilde\Theta_Q\left(v^{-2}\right)=\prod_{i=1}^n\tilde\theta\left((q_iv)^{-2}\right)=2^{-n}\prod_{i=1}^n\psi(q_iv),\quad \text{and}\quad  W_Q(v^{-2})=2^{-n}\prod_{i=1}^n(q_iv-1).\]
\end{proof}

As a corollary to this lemma we obtain the expressions for the Casimir energy of the two types of boxes which shall be used to prove Theorem \ref{thm:no_max_low}.

\begin{cor} \label{cor:casimir_2boxes}
For the unit equilateral box in $\R^m$, with associated matrix $I_m$ the $m \times m$ identity matrix, we have 
\beq \zeta_\scD \left(-\frac 1 2, I_m \right) = 2^{-m} \int_0 ^\infty \left[(v-1)^m - \psi(v)^m \right] dv. \label{eq:I0m} \eeq 
For $\eta \in (0, 1)$, let $\mathcal B_\eta \in \mathcal S_{m+1}$ have $m$ sides of length $\eta^{\frac 1{m+1}}$ and one side of length $\eta^{- \frac m{m+1}}$, and let $Q_\eta$ be the associated matrix.  Then 
\beq \zeta_\scD\left(-\frac 1 2, Q_\eta\right) =  \frac{1}{2^{m+1} \eta^{\frac{m+2}{m+1}}} \int_0 ^\infty \left[ (x-1)^m (x-\eta)-\psi(x)^m \eta \psi(x/\eta) \right] dx. \label{eq:Qeta} \eeq 
Moreover, 
\beq &\int_0 ^\infty \left[ (x-1)^m (x-\eta)-\psi(x)^m \eta \psi(x/\eta) \right] dx \nn \\  &= \int_0 ^\infty x\left[(x-1)^m - \psi(x)^m \right] dx + \mathcal O (\eta), \textrm{ as } \eta \to 0. \label{eq:eta_0_int} \eeq 
\end{cor}

\begin{proof}
The expression \eqref{eq:I0m} follows immediately from Lemma \ref{lem:zeta_half} since $q_i=1$ for all $i$.  For the expression \eqref{eq:Qeta}, we make the substitution $v = x \eta^{-\frac 1 {m+1}}$ in Lemma \ref{lem:zeta_half}.  For the limiting expression in \eqref{eq:eta_0_int} as $\eta \to 0$, we write the integrand as 
\[(x-1)^m (x-\eta) - \psi(x)^m \eta \psi(x/\eta) =  (x-1)^m(x -\eta \psi(x/\eta) -  \eta) - (\psi(x)^m - (x-1)^m)\eta \psi(x/\eta).\]
Then, since $\psi(v)=\theta(v^{-2})-1=v \theta(v^2)-1$ by the Jacobi identity \eqref{eq:poisson_nice}, we have $\eta\psi(x/\eta)=x-\eta+2x\tilde\theta(x^2/\eta^2)$, so
\[ (x-1)^m(x -\eta \psi(x/\eta) -  \eta) = -(x-1)^m 2x \tilde \theta(x^2/\eta^2). \]
The substitution $x=\eta y$ gives
\[
\int_0^\infty (x-1)^m 2x \tilde\theta(x^2/\eta^2) dx=\eta^2\int_0^\infty(\eta y-1)^m 2y \tilde\theta(y^2) dy=\mathcal O(\eta^2),\qquad \eta\to0,
\]
since $|\eta y-1|\le1+y$ for $0<\eta\le1$, and $y \tilde\theta(y^2)$ is bounded and decays exponentially. For the second term we use $\eta\psi(x/\eta)=x-\eta+2x\tilde\theta(x^2/\eta^2)$. The term $x$ gives the limit; the term $-\eta$ gives $\mathcal O(\eta)$, since $\psi^m-(x-1)^m$ is integrable; and the last term gives $\mathcal O(\eta^2)$ by the same substitution, since $\psi^m-(x-1)^m$ is bounded. Hence
\[
-\int_0^\infty\bigl(\psi(x)^m-(x-1)^m\bigr)\eta\psi(x/\eta) dx=\int_0^\infty x\bigl[(x-1)^m-\psi(x)^m\bigr]dx+\mathcal O(\eta).
\]
\end{proof}

Consider $ \zeta_\scD(-\frac 1 2, I_m)$.  We have 
\[ \int_0 ^1 (v-1)^m dv = \frac{(-1)^m}{m+1}, \quad m \geq 1.\]
Since $\psi(v)>0$ for all $v>0$, this shows that the part of the integral from $0$ to $1$ is always negative when $m$ is odd.  Moreover, for $v>1$, $v-1<\psi(v)$, so for all $m$ (even or odd) the integral from $1$ to $\infty$ is always negative.  Hence, we see that $ \zeta_\scD(-\frac 1 2, I_m) < 0$ for all odd $m \geq 1$.  For even $m$, there is a competition between terms.  We will prove that $ \zeta_\scD(-\frac 1 2, I_m) > 0$ for all even $m=2, \ldots, 28$.  

\begin{prop} \label{prop:casimir_pos_odd}
For every even integer $m$ with $2\le m\le 28$,
\[
\int_0^\infty \bigl[(v-1)^m-\psi(v)^m\bigr] dv>0 .
\]
\end{prop}

\begin{proof}
The function $\psi(v)$ is a sum over the integers $k \geq 1$ of the terms $e^{-\pi k^2/v^2}$, each of which are increasing in $v$, so $\psi$ is increasing on $(0,\infty)$. Its value at $v=1$ is
\[
\psi(1)=\frac{\pi^{1/4}}{\Gamma(3/4)}-1=0.08643\ldots<0.0865 .
\]
We then have the first estimate 
\[
\int_0^1\bigl[(v-1)^m-\psi(v)^m\bigr]dv\;\ge\;\frac{(-1)^m}{m+1}-\psi(1)^m = \frac 1{m+1} - \psi(1)^m, \textrm{ for $m$ even.}
\]
By the Jacobi transformation 
we have
\begin{equation}\label{eq:jacobi}
\psi(v)=v-1+2v \tilde \theta(v^2), \qquad r(v) =2v \tilde \theta(v^2) = 2v\sum_{k\ge1}e^{-\pi k^2v^2}>0 .
\end{equation}
For $v\ge1$ the function $r$ is decreasing, and $r(1) = \psi(1)$. 
Moreover, since $k^2-1\ge 3(k-1)$ for all integers $k\ge1$,
\begin{equation}\label{eq:rbound}
r(v)\le 2v e^{-\pi v^2}\sum_{j\ge0}e^{-3\pi j}=\frac{2v e^{-\pi v^2}}{1-e^{-3\pi}}\le 2.0002 v e^{-\pi v^2},\qquad v\ge1.
\end{equation}
In particular $r(3/2)\le 3.0003 e^{-9\pi/4}<0.0026$.

For $v\ge1$ we have $0\le v-1\le\psi(v)$ by \eqref{eq:jacobi}. The mean value theorem therefore gives
\begin{equation}\label{eq:mvt}
0\le \psi(v)^m-(v-1)^m\le m r(v) \psi(v)^{m-1}.
\end{equation}

On $[1,\tfrac32]$ we have $r\le \psi(1)$, and $\psi(v)\le\psi(3/2)=\tfrac12+r(3/2)\le\tfrac12+\psi(1)$. Hence
\[
\int_1^{3/2}\bigl[\psi(v)^m-(v-1)^m\bigr]dv\le \frac{m\psi(1)}{2}\Bigl(\frac12+\psi(1)\Bigr)^{m-1}.
\]

On $[\tfrac32,\infty)$ we have $v-1\ge\frac12$ and $r(v)\le r(3/2)<0.0026$. So
\[
\psi(v)=(v-1)\Bigl(1+\frac{r(v)}{v-1}\Bigr)\le 1.0052 (v-1).
\]
Inserting this and \eqref{eq:rbound} into \eqref{eq:mvt}, then extending the range of integration to $[1,\infty)$ (the integrand is nonnegative) and substituting $v=1+u$, we obtain
\[
\int_{3/2}^\infty\bigl[\psi(v)^m-(v-1)^m\bigr]dv\le 2.0002 m (1.0052)^{m-1}\int_0^\infty u^{m-1}(1+u) e^{-\pi(1+u)^2} du .
\]

\begin{claim}
For every $u_0\ge0$, with $\lambda=2\pi(1+u_0)$,
\[
\int_0^\infty u^{m-1}(1+u) e^{-\pi(1+u)^2} du\le e^{\pi(u_0^2-1)}\left[\frac{(m-1)!}{\lambda^m}+\frac{m!}{\lambda^{m+1}}\right].
\]
\end{claim}

\begin{proof}[Proof of the claim]
By convexity, $(1+u)^2$ lies above its tangent line at $u_0$:
\[
(1+u)^2\ge(1+u_0)^2+2(1+u_0)(u-u_0)=\lambda u/\pi+1-u_0^2 .
\]
Hence $e^{-\pi(1+u)^2}\le e^{\pi(u_0^2-1)}e^{-\lambda u}$. The claim follows from $\int_0^\infty u^{n}e^{-\lambda u} du=n!/\lambda^{n+1}$.
\end{proof}

Combining the three estimates, the integral is positive whenever
\beq 
\frac1{m+1}-\psi(1)^m &>& \frac{m\psi(1)}{2}\Bigl(\frac12+\psi(1)\Bigr)^{m-1} \nn \\ 
&+&2.0002 m (1.0052)^{m-1}e^{\pi(u_0^2-1)}\left[\frac{(m-1)!}{\lambda^m}+\frac{m!}{\lambda^{m+1}}\right] \label{eq:criterion} 
\eeq 
for some $u_0\ge0$, where $\lambda=2\pi(1+u_0)$.  Both sides are explicit. Replacing $\psi(1)$ by $0.0865$ decreases the left side and increases the right side, so it is enough to check \eqref{eq:criterion} with $\psi(1)=0.0865$. We choose $u_0$ close to the root of $2\pi u_0(1+u_0)=m-1$. This is where the tangent exponential has the same maximizer as the original integrand, so little is lost.

The table lists the chosen $u_0$, the left side rounded down, and the right side rounded up, all to four decimals. In each row the left side is larger, which proves the proposition.
\[
\begin{array}{cccc}
\toprule 
m & u_0 & \text{LHS of \eqref{eq:criterion}} \ge & \text{RHS of \eqref{eq:criterion}} \le\\
\midrule 
2 & 0.14 & 0.3258 & 0.0554 \\
4 & 0.35 & 0.1999 & 0.0358 \\
6 & 0.52 & 0.1428 & 0.0184 \\
8 & 0.67 & 0.1111 & 0.0085 \\
10 & 0.80 & 0.0909 & 0.0037 \\
12 & 0.91 & 0.0769 & 0.0016 \\
14 & 1.02 & 0.0666 & 0.0008 \\
16 & 1.12 & 0.0588 & 0.0005 \\
18 & 1.22 & 0.0526 & 0.0005 \\
20 & 1.31 & 0.0476 & 0.0008 \\
22 & 1.40 & 0.0434 & 0.0016 \\
24 & 1.48 & 0.0399 & 0.0036 \\
26 & 1.56 & 0.0370 & 0.0094 \\
28 & 1.63 & 0.0344 & 0.0269 \\
\bottomrule
\end{array}
\]
\end{proof}

We now consider $\zeta_\scD(-\frac 1 2, Q_\eta)$, or more precisely the integral in \eqref{eq:eta_0_int}.  We establish the values of $m$ for which this integral is positive.  

\begin{prop} \label{prop:cas_pos_even} 
For every even integer $m$ with $2\le m\le 20$,
\[
\int_0^\infty x\bigl[(x-1)^m-\psi(x)^m\bigr] dx>0 .
\]
\end{prop}

\begin{proof}
We estimate the part of the integral from $0$ to $1$.  Here $0\le\psi\le\psi(1)$, and $m$ is even, so
\[
\int_0^1 x\bigl[(x-1)^m-\psi(x)^m\bigr]dx\;\ge\;\int_0^1x(1-x)^m dx-\frac{\psi(1)^m}{2}=\frac1{(m+1)(m+2)}-\frac{\psi(1)^m}{2}.
\]
We continue using the same notation as in the proof of Proposition \ref{prop:casimir_pos_odd}.  For $x\ge1$, with $\psi(x) = x-1+r(x)$, $r(x)$ is decreasing with $r(1)=\psi(1)$. Since $k^2-1\ge3(k-1)$ for all integers $k\ge1$,
\begin{equation}\label{eq:rbound_2}
r(x)\le\frac{2x e^{-\pi x^2}}{1-e^{-3\pi}}\le 2.0002 x e^{-\pi x^2},\qquad x\ge1.
\end{equation}
In particular $r(7/5)\le 2.8003 e^{-49\pi/25}<0.006$.

For $x\ge1$ we have $0\le x-1\le\psi(x)$ by \eqref{eq:jacobi}. The mean value theorem therefore gives
\begin{equation}\label{eq:mvt_2}
0\le \psi(x)^m-(x-1)^m\le m r(x) \psi(x)^{m-1}.
\end{equation}

On $[1,\tfrac75]$ we have $x\le\frac75$ and $r\le\psi(1)$. Also $\psi(x)\le\psi(7/5)=\frac25+r(7/5)\le\frac25+\psi(1)$. Hence
\[
\int_1^{7/5}x\bigl[\psi^m-(x-1)^m\bigr]dx\le \frac{14}{25} m \psi(1)\Bigl(\frac25+\psi(1)\Bigr)^{m-1}.
\]

On $[\tfrac75,\infty)$ we have $x-1\ge\frac25$ and $r(x)\le r(7/5)<0.006$. So
\[
\psi(x)=(x-1)\Bigl(1+\frac{r(x)}{x-1}\Bigr)\le 1.015 (x-1).
\]
Inserting this and \eqref{eq:rbound_2} into \eqref{eq:mvt_2}, then extending the range of integration to $[1,\infty)$ and substituting $x=1+u$, we obtain
\[
\int_{7/5}^\infty x\bigl[\psi^m-(x-1)^m\bigr]dx\le 2.0002 m (1.015)^{m-1}\int_0^\infty u^{m-1}(1+u)^2e^{-\pi(1+u)^2} du .
\]

\begin{claim}
For every $u_0\ge0$, with $\lambda=2\pi(1+u_0)$,
\[
\int_0^\infty u^{m-1}(1+u)^2e^{-\pi(1+u)^2}du\le e^{\pi(u_0^2-1)}\left[\frac{(m-1)!}{\lambda^m}+\frac{2 m!}{\lambda^{m+1}}+\frac{(m+1)!}{\lambda^{m+2}}\right].
\]
\end{claim}

\begin{proof}[Proof of the claim]
By convexity, $(1+u)^2\ge(1+u_0)^2+2(1+u_0)(u-u_0)=1-u_0^2+\lambda u/\pi$. Hence $e^{-\pi(1+u)^2}\le e^{\pi(u_0^2-1)}e^{-\lambda u}$. We expand $(1+u)^2$ and use $\int_0^\infty u^ne^{-\lambda u}du=n!/\lambda^{n+1}$ to complete the proof. 
\end{proof}

The integral is positive whenever, for some $u_0\ge0$ with $\lambda=2\pi(1+u_0)$,
\begin{multline}\label{eq:criterion_2}
\frac1{(m+1)(m+2)}-\frac{\psi(1)^m}{2}\;>\;\frac{14}{25} m \psi(1)\Bigl(\frac25+\psi(1)\Bigr)^{m-1}\\
+2.0002 m (1.015)^{m-1}e^{\pi(u_0^2-1)}\left[\frac{(m-1)!}{\lambda^m}+\frac{2 m!}{\lambda^{m+1}}+\frac{(m+1)!}{\lambda^{m+2}}\right].
\end{multline}
Replacing $\psi(1)$ by $0.0865$ decreases the left side and increases the right side, so it suffices to check \eqref{eq:criterion_2} with $\psi(1)$ replaced by $0.0865$.

The table lists the chosen $u_0$, the left side rounded down, and the right side rounded up, all to six decimals. For $m\ge4$, $u_0$ is the root of $\frac{m-1}{u}+\frac{2}{1+u}=2\pi(1+u)$ (the maximizer of the integrand in the claim) rounded to two decimals; for $m=2$ we take $u_0=0$. In each row the left side is larger, which proves the proposition.
\[
\begin{array}{cccc}
\toprule
m & u_0 & \text{LHS of \eqref{eq:criterion_2}} \ge & \text{RHS of \eqref{eq:criterion_2}} \le\\
\midrule
2 & 0.00 & 0.079592 & 0.055082 \\
4 & 0.41 & 0.033305 & 0.023600 \\
6 & 0.58 & 0.017856 & 0.008452 \\
8 & 0.72 & 0.011111 & 0.002828 \\
10 & 0.85 & 0.007575 & 0.001012 \\
12 & 0.97 & 0.005494 & 0.000497 \\
14 & 1.08 & 0.004166 & 0.000426 \\
16 & 1.18 & 0.003267 & 0.000576 \\
18 & 1.27 & 0.002631 & 0.000992 \\
20 & 1.36 & 0.002164 & 0.001979 \\
\bottomrule
\end{array}
\]
\end{proof}

\begin{proof}[Proof of Theorem \ref{thm:no_max_low}]
By Corollary \ref{cor:casimir_2boxes} and Proposition \ref{prop:casimir_pos_odd} we have that $\zeta_\scD(-\frac 1 2, I_m)>0$ for even values of $m=2, \ldots, 28$, where $I_m$ is the $m\times m$ identity matrix.  Consequently, the $m+1$ dimensional box $\mathcal B_{m+1} = L^{-\frac 1 m} \mathcal B \times [0, L]$, with associated matrix $Q_{m+1}$, and with $\mathcal B$ the unit-equilateral $m$-dimensional box has  spectral height 
\[\zeta_\scD'(0, Q_{m+1}) = -\zeta_\scD(-\frac 1 2, I_m) L^{\frac{m+1}{m}} + \text{l.o.t.} \textrm{ as } L \to \infty. \]
Here, l.o.t. are lower order terms.  Consequently the zeta-regularized determinant tends to $+\infty$ as $L \to \infty$.  This holds for the odd dimensions $m+1=3, \ldots, 29$.  

Next, fix $m$ even with $2 \leq m \leq 20$.  By Corollary \ref{cor:casimir_2boxes} and Proposition \ref{prop:cas_pos_even}, there exists $\eta > 0$ sufficiently small such that 
\[ \zeta_\scD(-\frac 1 2, Q_\eta) > 0.\]
Here $Q_\eta$ is the matrix associated with $\mathcal B_\eta \in \mathcal S_{m+1}$ as in Corollary \ref{cor:casimir_2boxes}.  Then, the $m+2$ dimensional box $\mathcal B_{m+2} = L^{-\frac 1 {m+1}} \mathcal B_\eta \times [0, L]$ similarly has spectral height tending to $-\infty$ as $L \to \infty$, and therewith its zeta-regularized determinant tends to $+\infty$.  Consequently, the determinant is unbounded on $\mathcal S_{m+2}$ for even values of $m$ from $2$ to $20$, corresponding to the dimensions $4$ to $22$. 
\end{proof}

\begin{remark}
     We know from Theorem \ref{thm_global_max} that in all dimensions 40 and higher, the determinant has a global maximum.  It is not entirely clear what happens in the even dimensions from 24 to 38 as well as the odd dimensions from 31 to 39. There could be some other type of boxes for which the Casimir energy is positive and which can be used to show that the determinant has no global maximum as in the proof of Theorem \ref{thm:no_max_low}.  It could also be the case that no such boxes exist.  This may be the subject of further exploration. 
\end{remark}

\subsection{Local extrema analysis}\label{s:local}
In this subsection we study the determinant near the equilateral box, that is, near $Q=I_n$, and it is organized in the following way. We first compute the Hessian of the Epstein zeta function at the identity  in Lemma \ref{dir_beta} . Using the combinatorial expression of Corollary \ref{cor:box_unit}, we then show that the Hessian of the Dirichlet spectral height at the equilateral box is a scalar multiple $h_n$ of the identity, and we express $h_n$ through a positive function $M$ of a real variable in Proposition \ref{prop:box_hess}. The local behavior is therefore determined by the sign of $h_n$: in even dimensions it is always positive, which proves the even case of Theorem \ref{th:local_extremum}, while in odd dimension it is the sign of $M(n)-1$. The rest of the subsection is devoted to the odd case, which requires a finer analysis of $M$.

We will use the coordinates $x_i=\log q_i$, so that $Q=\diag(e^{x_1},\dots,e^{x_n})$. With this choice the volume constraint reads
\[\prod_{i=1}^nq_i=1\qquad\Leftrightarrow\qquad \sum_{i=1}^nx_i=0.\]
The boxes with unit volume are parametrized by \[\Sigma_n=\{x\in\mathbb{R}^n:\sum_i x_i=0\},\] and the equilateral box corresponds to $x=0$. There are two reasons for this choice. First, $\Sigma_n$ is a vector space, so the curves along which we take the variations, $\tau\mapsto\tau v$ with $v\in\Sigma_n$, are lines. For a smooth function $f$ and a curve $\gamma$
\[\frac{d^2}{d\tau^2}f\bigl(\gamma(\tau)\bigr)=D^2f\bigl(\gamma(\tau)\bigr)\bigl[\gamma'(\tau),\gamma'(\tau)\bigr]+\nabla f\bigl(\gamma(\tau)\bigr)\cdot\gamma''(\tau),\]
and the second term vanishes along lines. In the coordinates $q_i$ the constraint is curved, the second term does not vanish in general at a critical point of the constrained problem, and we would need the Lagrange multiplier to take into account the curvature of the constraint. Second, the rescaling $\tilde Q_J=c_JQ_J$ of a sub-box to unit volume in Corollary \ref{cor:box_unit} becomes the linear map $x_J\mapsto(x_i-\bar x_J)_{i\in J}$, where $\bar x_J$ is the mean of the coordinates in $J$.

We also introduce, with this choice of coordinates $x\in\R^n$ and $t>0$, the multidimensional theta function, given by 
\[\Theta_x(t)=\sum_{k\in\Z^n}e^{-\pi Q^{-2}[k]t}=\prod_{i=1}^n\theta\big(t e^{-2x_i}\big),\qquad Q=\diag(e^{x_1},\dots,e^{x_n}).\]

\begin{lem}\label{lem:dir_integral}
    Let $g(t)=\log\theta(t)$, where we recall $\theta(t) = \sum_{k\in \Z} e^{-\pi k^2 t}$.  Define
    \[T(t)=t\frac{\theta'}{\theta}(t)+t^{2}\Big(\frac{\theta'}{\theta}\Big)'(t)=tg'(t)+t^{2}g''(t).\]
    Then, $T(t)>0$ for every $t>0$, and
\beq tg'(t)+t^{-1}g'(1/t)&=&-1/2, \label{eq:i} \\
T(t)&=&T(1/t), \label{eq:ii} \\
\int_{1}^\infty T(t)  \frac{dt}{t}&=&\frac{1}{4}. \label{eq:iii}
\eeq
\end{lem}
\begin{proof}
The positivity of $T$ is proved in \cite{Montgomery1988}, Lemma 3. The logarithm in the Jacobi identity gives
\[g(1/t)=g(t)+\frac{1}{2}\log(t).\]
Taking the derivative and multiplying both sides by $t>0$ demonstrates \eqref{eq:i}.  Differentiating \eqref{eq:i} and using the definition of $T$ shows \eqref{eq:ii}.  We then observe that $\frac{T(t)}{t} = (tg'(t))'$. As $t\to\infty$, $\theta(t)=1+2e^{-\pi t}+\mathcal{O}\left(e^{-4\pi t}\right)$ thus $tg'(t)\to 0$. Hence
\[\int_1^\infty T(t)\frac{dt}{t}=\lim_{t\to\infty}tg'(t)-g'(1)=-\frac{\theta'(1)}{\theta(1)} .\]
Applying \eqref{eq:i} with $t=1$ then proves \eqref{eq:iii}.
\end{proof}
We now show that $x=0\in\Sigma_n$ is a critical point of $\zeta_E'(0)$, for any $n\geq2$, and we establish an expression for the Hessian at this critical point. This result holds for any dimension $n\geq2$, in particular both even and odd dimensions.
\begin{lem}\label{dir_beta}
For $n\ge2$, the function \[x\mapsto\zeta_E'\big(0,\diag(e^{x})^{-2}\big)\] is smooth on $\Sigma_n$, and $x=0$ is a critical point. The Hessian at this critical point is equal to $\beta_n$ times the identity on $\Sigma_n$ and
\begin{equation}\label{eq:dir_torus_hess}
\beta_n =4\int_1^\infty T(t) \theta(t)^n\left(1+t^{n/2}\right)\frac{dt}{t}.
\end{equation}
\end{lem}

\begin{proof}
By Riemann's Lemma \ref{riemann}, for $x\in\Sigma_n$ and $s\in\C$
\[\begin{aligned}
\zeta_E\big(s,\diag(e^{x})^{-2}\big)&=\frac{\pi^s}{\Gamma(s)}\left(\frac{1}{s-n/2}-\frac{1}{s}\right)\\
&\quad+\frac{\pi^s}{\Gamma(s)}\int_1^\infty \left[\left(\Theta_x(t)-1\right)t^{s-1}+\left(\Theta_{-x}(t)-1\right) t^{n/2-s-1}\right] dt.
\end{aligned}\]
Hence,
\begin{equation}\label{eq:dir_eps}
\zeta_E'\big(0,\diag(e^{x})^{-2}\big)=C_n + \int_1^\infty\Big[\big(\Theta_x(t)-1\big)+t^{n/2}\big(\Theta_{-x}(t)-1\big)\Big]\frac{dt}{t},
\end{equation}
where $C_n$ does not depend on $x$. Next, fix $v\in\Sigma_n$ and $t>0$, and consider the smooth map $\tau\mapsto\Theta_{\tau v}(t)$, $\tau\in\R$. We define $t_i(\tau)=t e^{-2\tau v_i}$, so that
\[t_i'=-2v_it_i,\]
where, here and below, primes on $t_i$ and $\ell$ denote $d/d\tau$. Moreover, with $g=\log\theta$ as in Lemma \ref{lem:dir_integral}, we define $\ell(\tau)=\log\Theta_{\tau v}(t)=\sum_{i}g(t_i(\tau))$. Then
\[\begin{aligned}
\ell'(\tau)&=\sum_{i}g'(t_i) t_i'=-2\sum_i v_i t_i g'(t_i),\\
\ell''(\tau)&=-2\sum_i v_i t_i'\big(g'(t_i)+t_ig''(t_i)\big)=4\sum_i v_i^2\big(t_ig'(t_i)+t_i^2g''(t_i)\big)=4\sum_i v_i^2 T(t_i).
\end{aligned}\]
Setting $\tau=0$ implies $t_i(0)=t$ and, since $\sum_i v_i=0$, we get
\[\ell(0)=n g(t),\qquad \ell'(0)=-2t g'(t)\sum_i v_i=0,\qquad \ell''(0)=4 T(t)\sum_i v_i^2=4||v||^2 T(t).\]
Now, $\Theta_{\tau v}(t)=e^{\ell(\tau)}$, thus $\frac{d}{d\tau}\Theta_{\tau v}(t)=\ell'(\tau)e^{\ell(\tau)}$ and
\[\frac{d^2}{d\tau^2}\Theta_{\tau v}(t)=\left(\ell''(\tau)+\ell'(\tau)^2\right)e^{\ell(\tau)},\]
and at $\tau=0$ we have $e^{\ell(0)}=\theta(t)^n$. Moreover, since $\Theta_{-\tau v}(t)=\Theta_{\tau(-v)}(t)$ with $-v\in\Sigma_n$ and $|{-v}|=|v|$, the following holds
\[\frac{d}{d\tau}\Theta_{\pm\tau v}(t)\Big|_{\tau=0}=0,\qquad \frac{d^2}{d\tau^2}\Theta_{\pm\tau v}(t)\Big|_{\tau=0}=4 \|v\|^2 T(t) \theta(t)^n.\]
Finally, differentiating \eqref{eq:dir_eps} with $\tau v\in\Sigma_n$ in place of $x$ under the integral sign, we get
\[\begin{aligned}
\frac{d}{d\tau}\zeta_E'\big(0,\diag(e^{\tau v})^{-2}\big)\Big|_{\tau=0}&=0,\\
\frac{d^2}{d\tau^2}\zeta_E'\big(0,\diag(e^{\tau v})^{-2}\big)\Big|_{\tau=0}&=4\|v\|^2\int_1^\infty T(t) \theta(t)^n\big(1+t^{n/2}\big)\frac{dt}{t}=\beta_n\|v\|^2 .
\end{aligned}\]
Since $v\in\Sigma_n$ is arbitrary, the first identity shows that $0$ is a critical point of $\zeta_E'\big(0,\diag(e^{x})^{-2}\big)$, while the second shows that the Hessian is equal to $\beta_n$ times the identity.
\end{proof}

The Dirichlet spectral height is related to the Epstein zeta function via the combinatorial expression given in Corollary \ref{cor:box_unit}.  For this reason we introduce the following functions which will then be essential to our analysis of the Hessian.

\begin{lem}\label{dir_sum}
For $n\ge2$ and $y\in\R$ let
\[K_n(y)=\int_1^yz(z-1)^{n-2} dz.\]
Then, for every $y\in\R$,
\begin{equation}\label{eq:sum_K}
\sum_{m=2}^n(-1)^{n-m}\binom{n-2}{m-2}\frac{y^m}{m}=K_n(y)+\frac{(-1)^n}{n(n-1)}.
\end{equation}
Moreover, $K_n(y)>0$ for every $y>1$.
\end{lem}
\begin{proof}
Taking the derivative of the left-hand side of \eqref{eq:sum_K} and setting $j=m-2$, we get
\[\sum_{m=2}^{n}(-1)^{n-m}\binom{n-2}{m-2}y^{m-1}=y\sum_{j=0}^{n-2}\binom{n-2}{j}(-1)^{n-2-j}y^{j}=y(y-1)^{n-2}=K_n'(y).\]
Hence $K_n(y)$ differs from the left side of \eqref{eq:sum_K} by a constant. At $y=0$ the left-hand side vanishes, while by \eqref{eq:eulerbeta}
\[K_n(0)=-\int_0^1z(z-1)^{n-2} dz=(-1)^{n-1}\int_0^1z(1-z)^{n-2} dz=(-1)^{n-1}\mathrm B(2,n-1)=\frac{(-1)^{n-1}}{n(n-1)}.\]
This proves \eqref{eq:sum_K}. Finally, for $y>1$ the integrand $z(z-1)^{n-2}$ is positive on $(1,y)$, thus $K_n(y)>0$.
\end{proof}

Next, we define the function that determines the sign of the Hessian. Multiplying \eqref{eq:i} by $2\theta(t)$ and using the Jacobi identity, we get
\begin{equation}\label{eq:V_jacobi}
\theta(t)+2t\theta'(t)=-2t^{-3/2} \theta'(1/t)>0,\qquad t>0,
\end{equation}
since $\theta'<0$. For $y\ge2$ we define
\begin{equation}\label{eq:M_def}
M(y)=y(y-1)\int_0^\infty\bigl(\theta(t)-1\bigr)^{y-2}\bigl(\theta(t)+2t\theta'(t)\bigr)\bigl(-\theta'(t)\bigr) dt.
\end{equation}
By \eqref{eq:V_jacobi} the integrand is positive, thus $M(y)\in(0,+\infty]$; we will show that $M(y)$ is finite.

\begin{prop}\label{prop:box_hess}
Let $n\ge2$. The Dirichlet spectral height $x\mapsto\zeta'_\scD\bigl(0,\diag(e^{x})\bigr)$ is smooth on $\Sigma_n$, $x=0$ is a critical point, and the Hessian at this critical point is equal to $h_n$ times the identity on $\Sigma_n$, where
\begin{equation}\label{eq:key}
h_n=\frac{2^{1-n}}{n-1}\Big[(-1)^n+M(n)\Big].
\end{equation}
Moreover $M(n)$ is finite, and
\begin{equation}\label{eq:M_integer}
M(n)=2n(n-1)\int_1^\infty T(t)\Big[K_n\big(\theta(t)\big)+K_n\big(\sqrt t \theta(t)\big)\Big]\frac{dt}{t}.
\end{equation}
\end{prop}
\begin{proof}
By Corollary \ref{cor:box_unit}
\begin{equation}\label{eq:dir_tori}
\zeta'_{\scD}(0, Q)= -2\log\pi\left(-\frac{1}{2}\right)^n + 2^{-n}
\sum_{m=1}^{n} (-1)^{ n-m}
\sum_{\substack{J \subseteq \{1,\dots,n\} \\ |J| = m}}
\zeta'_E\big(0, \tilde{Q}_J^{-2}\big),
\end{equation}
where $\tilde Q_J=c_JQ_J=\diag\bigl(e^{x_i-\bar x_J}\bigr)_{i\in J}$, with $Q_J=\diag(q_i)_{i\in J}$, $c_J=\bigl(\prod_{i\in J}q_i\bigr)^{-1/|J|}=e^{-\bar x_J}$ and $\bar x_J=\frac{1}{|J|}\sum_{i\in J}x_i$, is the matrix of the sub-box rescaled to unit volume. Fix $v\in\Sigma_n$ and set $x=\tau v$, $\tau\in\R$. Then
\[
Q(\tau)=\diag\bigl(e^{\tau v_i}\bigr)_{i=1}^n,\qquad
\tilde Q_J(\tau)=\diag\bigl(e^{\tau(v_i-\bar v_J)}\bigr)_{i\in J},\qquad
\bar v_J=\frac1{|J|}\sum_{i\in J}v_i .
\]
In general $v_J=(v_i)_{i\in J}$ does not belong to $\Sigma_J=\{y\in\R^J:\sum_{i\in J}y_i=0\}$, where Lemma \ref{dir_beta} applies; however, $\tilde Q_J(\tau)$ depends on $\tau$ only through the orthogonal projection of $v_J$ onto $\Sigma_J$,
\[
P_Jv=(v_i-\bar v_J)_{i\in J}\in\Sigma_J .
\]
Hence, for $|J|=m\ge2$, $\tau\mapsto\zeta'_E\bigl(0,\tilde Q_J(\tau)^{-2}\bigr)$ is the function of Lemma \ref{dir_beta}, in dimension $m$, along the line $\tau\mapsto\tau P_Jv$, and by the same lemma
\[
\frac{d}{d\tau}\zeta'_E\bigl(0,\tilde{Q}_J(\tau)^{-2}\bigr)\Big|_{\tau=0}=0,\qquad
\frac{d^2}{d\tau^2}\zeta'_E\bigl(0,\tilde{Q}_J(\tau)^{-2}\bigr)\Big|_{\tau=0}=\beta_m\|P_Jv\|^2 ,
\]
while for $|J|=1$ the term is constant. In particular, by \eqref{eq:dir_unit} the spectral height is smooth on $\Sigma_n$, its derivative along every $v\in\Sigma_n$ vanishes at $x=0$, and
\begin{equation}\label{eq:dir_hess_v}
\frac{d^{2}}{d\tau^{2}}\zeta_\scD'\bigl(0,Q(\tau)\bigr)\Big|_{\tau=0}=2^{-n}\sum_{m=2}^n(-1)^{n-m} \beta_{m}\sum_{|J|=m}\|P_Jv\|^{2}.
\end{equation}
The inner sums are computed as follows. For $|J|=m$ we have $\|P_Jv\|^2=\sum_{i\in J}v_i^2-\frac1m\bigl(\sum_{i\in J}v_i\bigr)^2$. Each index belongs to $\binom{n-1}{m-1}$ subsets $J$ with $|J|=m$, and each pair of distinct indices to $\binom{n-2}{m-2}$ of them. Since $\sum_{i\ne k}v_iv_k=\bigl(\sum_iv_i\bigr)^2-\|v\|^2=-\|v\|^2$, we get
\[\begin{aligned}
\sum_{|J|=m}\sum_{i\in J}v_i^2&=\binom{n-1}{m-1}\|v\|^2,\\
\sum_{|J|=m}\Bigl(\sum_{i\in J}v_i\Bigr)^2&=\left[\binom{n-1}{m-1}-\binom{n-2}{m-2}\right]\|v\|^2=\binom{n-2}{m-1}\|v\|^2 .
\end{aligned}\]
Hence, using $\binom{n-1}{m-1}=\binom{n-2}{m-1}+\binom{n-2}{m-2}$ and $(m-1)\binom{n-2}{m-1}=(n-m)\binom{n-2}{m-2}$,
\[
\sum_{|J|=m}\|P_Jv\|^{2}=\left[\binom{n-2}{m-2}+\frac{m-1}{m}\binom{n-2}{m-1}\right]\|v\|^2=\binom{n-2}{m-2}\frac nm \|v\|^2 .
\]
Substituting in \eqref{eq:dir_hess_v}, we obtain that the second derivative in \eqref{eq:dir_hess_v} is equal to $h_n|v|^2$ for every $v\in\Sigma_n$, where
\begin{equation}\label{eq:dir_hess_box}
h_n=\frac{n}{2^{n}}\sum_{m=2}^{n}(-1)^{n-m}\binom{n-2}{m-2}\frac{\beta_m}{m},
\end{equation}
that is, the Hessian is equal to $h_nI$.

We now compute $h_n$. We substitute \eqref{eq:dir_torus_hess} in \eqref{eq:dir_hess_box}. Using $\theta(t)^m\bigl(1+t^{m/2}\bigr)=\theta(t)^m+\bigl(\sqrt t \theta(t)\bigr)^m$, and exchanging the summation and the integral, we get
\[
h_n=\frac{4n}{2^{n}}\int_1^\infty T(t)\sum_{m=2}^{n}(-1)^{n-m}\binom{n-2}{m-2}\frac{\theta(t)^m+\bigl(\sqrt t \theta(t)\bigr)^m}{m} \frac{dt}{t}.
\]
By Lemma \ref{dir_sum}, applied with $y=\theta(t)$ and $y=\sqrt t \theta(t)$, the sum is equal to
\[K_n\big(\theta(t)\big)+K_n\big(\sqrt t \theta(t)\big)+\frac{2(-1)^n}{n(n-1)},\]
and by \eqref{eq:iii} the constant term contributes
\[\frac{4n}{2^n}\cdot\frac{2(-1)^n}{n(n-1)}\int_1^\infty T(t)\frac{dt}{t}=\frac{4n}{2^n}\cdot\frac{2(-1)^n}{n(n-1)}\cdot\frac14=(-1)^n\frac{2^{1-n}}{n-1}.\]
Hence
\[h_n=(-1)^n\frac{2^{1-n}}{n-1}+\frac{4n}{2^n}\int_1^\infty T(t)\Big[K_n\big(\theta(t)\big)+K_n\big(\sqrt t \theta(t)\big)\Big]\frac{dt}{t},\]
which is \eqref{eq:key}, provided that $M(n)$ satisfies \eqref{eq:M_integer}.

It remains to show that the right-hand side of \eqref{eq:M_integer} is equal to $M(n)$ as defined in \eqref{eq:M_def}. Using the Jacobi identity $\sqrt{t} \theta(t)=\theta(1/t)$ and the change of variable $s=1/t$ we get
\[\int_1^\infty T(t)K_n\big(\sqrt t \theta(t)\big)\frac{dt}{t}=\int_1^\infty T(t)K_n\big(\theta(1/t)\big)\frac{dt}{t}=\int_0^1 T(s)K_n\big(\theta(s)\big)\frac{ds}{s},\]
where we used \eqref{eq:ii}. Hence the integral in \eqref{eq:M_integer} is equal to
\[\int_0^\infty T(t) K_n\big(\theta(t)\big)\frac{dt}{t}.\]
By the definition of $T$, $T(t)/t=\bigl(tg'(t)\bigr)'=\frac{d}{dt}\bigl(tg'(t)+\frac12\bigr)$. We integrate by parts on $[\varepsilon,R]$. Since $K_n'(y)=y(y-1)^{n-2}$ and $\bigl(tg'+\frac12\bigr)\theta=\frac12\bigl(\theta+2t\theta'\bigr)$,
\[\int_\varepsilon^R T(t) K_n\big(\theta(t)\big)\frac{dt}{t}=\Big[\Big(tg'+\frac12\Big)K_n(\theta)\Big]_\varepsilon^R+\frac12\int_\varepsilon^R(\theta-1)^{n-2}\bigl(\theta+2t\theta'\bigr)(-\theta') dt.\]
The boundary terms vanish. Indeed, as $t\to\infty$, $tg'(t)\to0$ and $K_n(\theta(t))\to K_n(1)=0$. As $t\to0^+$, by \eqref{eq:i}
\[tg'(t)+\frac12=-t^{-1}g'(1/t)=t^{-1} \frac{|\theta'(1/t)|}{\theta(1/t)}=\mathcal O\bigl(t^{-1}e^{-\pi/t}\bigr),\]
while $0\le K_n(y)\le\int_1^yz^{n-1}dz\le y^n/n$ for $y\ge1$, and $\theta(t)\le\theta(1) t^{-1/2}$ for $t\le1$. Both integrands are positive, by Lemmas \ref{lem:dir_integral} and \ref{dir_sum} on the left and by \eqref{eq:V_jacobi} on the right. Hence, letting $\varepsilon\to0^+$ and $R\to\infty$,
\[\int_0^\infty T(t) K_n\big(\theta(t)\big)\frac{dt}{t}=\frac12\int_0^\infty(\theta-1)^{n-2}\bigl(\theta+2t\theta'\bigr)(-\theta') dt=\frac{M(n)}{2n(n-1)}.\]
This proves \eqref{eq:M_integer}. Since the left-hand side is finite, $M(n)$ is finite.
\end{proof}

Proposition \ref{prop:box_hess} shows that the dichotomy between even and odd dimensions comes only from the sign $(-1)^n$ in \eqref{eq:key}. Since $M>0$, $h_n>0$ for every even $n$, while for $n$ odd
\beq h_n>0\iff M(n)>1,\qquad h_n<0\iff M(n)<1. \label{eq:Mn_one} \eeq
In particular, the even case of Theorem \ref{th:local_extremum} now follows immediately: for $n$ even
\[h_n=\frac{2^{1-n}}{n-1}\bigl(1+M(n)\bigr)>0.\]
Since the gradient of the spectral height vanishes at $x=0$ and its Hessian is positive definite on $\Sigma_n$, the spectral height has a strict local minimum there, that is, the determinant has a strict local maximum at the equilateral box.
The rather surprising fact is the following.
\begin{prop}\label{prop:dir_three}
For $n=3$ we have $M(3)<\frac35$, thus $h_3<-\frac{1}{20}$. Consequently, among all orthogonal $3$-dimensional box domains with volume equal to one, the zeta-regularized determinant of the Laplacian with Dirichlet boundary conditions has a strict local minimum at the equilateral box with side lengths equal to one.
\end{prop}
\begin{proof}
By Proposition \ref{prop:box_hess}, $h_3=\frac18\bigl(M(3)-1\bigr)$ and
\[M(3)=12\int_1^\infty T(t)\Big[K_3\big(\theta(t)\big)+K_3\big(\sqrt t \theta(t)\big)\Big]\frac{dt}{t}.\]
We show that the integral is smaller than $\frac1{20}$. For $t\ge1$ we write $t=1+\sigma$ with $\sigma\ge0$. Since $e^\pi>2.7^3>19$
\[
\sum_{k\ge1}e^{-\pi kt}=\frac{1}{e^{\pi t}-1}\le\frac{e^{-\pi \sigma}}{e^\pi-1}<\frac{e^{-\pi \sigma}}{18},
\]
where we used that $e^{\pi t}-1\ge(e^\pi-1)e^{\pi \sigma}$. Moreover, since $k^2\ge k$
\begin{equation}\label{eq:dir_theta_bd}
0<\theta(t)-1=2\sum_{k\ge1}e^{-\pi k^2t}\le2\sum_{k\ge1}e^{-\pi kt}\le2\frac{e^{-\pi \sigma}}{e^\pi-1}<\frac{e^{-\pi \sigma}}{9}.
\end{equation}
We can bound $T(t)$. By Lemma \ref{lem:dir_integral}, $\frac{t\theta'+t^2\theta''}{\theta}=T+\bigl(\frac{t\theta'}{\theta}\bigr)^2>0$, and since $\theta\ge1$,
\[
T=\frac{t\theta'+t^2\theta''}{\theta}-\Big(\frac{t\theta'}{\theta}\Big)^2\le\frac{t\theta'+t^2\theta''}{\theta}\le t\theta'+t^2\theta'',
\]
and
\[t\theta'+t^2\theta''=2\sum_{k\ge1}\pi k^2t (\pi k^2t-1) e^{-\pi k^2t}.\]
All the terms in the series are positive. Moreover, $k\le2^{k-1}$ and $16<e^\pi$ give $k^4\le e^{\pi(k-1)}\le e^{\pi k(k-1)t}$, that is, $k^4e^{-\pi k^2t}\le e^{-\pi kt}$. Hence,
\[
T(t)\le2\pi^2t^2\sum_{k\ge1}k^4e^{-\pi k^2t}\le2\pi^2t^2\sum_{k\ge1}e^{-\pi kt}<\pi^2t^2\frac{e^{-\pi \sigma}}{9},
\]
and in particular
\[
 T(t)/t<\pi^2(1+\sigma)\frac{e^{-\pi \sigma}}{9}.
\]
By Lemma \ref{dir_sum}, $K_3(y)=\int_1^yz(z-1) dz$ is increasing on $[1,\infty)$. By \eqref{eq:dir_theta_bd}, $1<\theta(t)<\frac{10}{9}$. Since $\sqrt{1+\sigma}\le1+\frac \sigma2$ and $(1+\frac \sigma2)e^{-\pi \sigma}\le1$, \eqref{eq:dir_theta_bd} also gives
\[
1<\sqrt t \theta(t)<\Big(1+\frac \sigma2\Big)\Big(1+\frac{e^{-\pi \sigma}}{9}\Big)\le\frac{10}{9}+\frac \sigma2.
\]
Hence
\[
0<K_3\big(\theta(t)\big)+K_3\big(\sqrt t \theta(t)\big)<K_3\Big(\frac{10}{9}\Big)+K_3\Big(\frac{10}{9}+\frac \sigma2\Big)=\frac{p(\sigma)}{6},
\]
where
\[
p(\sigma)=\frac{29}{729}+\Big(\frac19+\frac \sigma2\Big)^2\Big(\frac{29}{9}+\sigma\Big)=\frac{\sigma^3}{4}+\frac{11}{12} \sigma^2+\frac{10}{27} \sigma+\frac{58}{729}.
\]
Since all the terms in the integral are positive, we can replace them with their upper bound. The integral is then dominated by
\[\begin{aligned}
\int_1^\infty T(t)\Big[K_3\big(\theta(t)\big)+K_3\big(\sqrt t \theta(t)\big)\Big]\frac{dt}{t}&<\int_0^\infty\left(\pi^2(1+\sigma)\frac{e^{-\pi \sigma}}{9}\right)\left[\frac{p(\sigma)}{6}\right]d\sigma\\
&=\frac{\pi^2}{54}\int_0^\infty(1+\sigma) p(\sigma) e^{-\pi \sigma} d\sigma\\
&=\frac{\pi^2}{54}\int_0^\infty\Big(\frac{\sigma^4}{4}+\frac76\sigma^3+\frac{139}{108}\sigma^2+\frac{328}{729}\sigma+\frac{58}{729}\Big)e^{-\pi \sigma}d\sigma\\
&=\frac1{54}\Big(\frac{6}{\pi^3}+\frac{7}{\pi^2}+\frac{139}{54 \pi}+\frac{328}{729}+\frac{58 \pi}{729}\Big)\\
&<\frac1{54}\Big(\frac29+\frac79+\frac{139}{162}+\frac{328}{729}+\frac{232}{729}\Big)<\frac1{20}.
\end{aligned}\]
Hence $M(3)<\frac{12}{20}=\frac35$, and $h_3<\frac18\bigl(\frac35-1\bigr)=-\frac{1}{20}$. The Hessian of the spectral height at $x=0$ is therefore negative definite on $\Sigma_3$, and the determinant has a strict local minimum at the equilateral box.
\end{proof}

Proposition \ref{prop:dir_three}, together with the even case, suggests that the local behavior of the determinant near the equilateral box alternates with the parity of $n$: a strict local maximum in even dimension and a strict local minimum in odd dimension. Once again the Dirichlet case behaves unexpectedly. The alternation holds in low dimension, but it stops in high dimension, where the equilateral box is a strict local maximum for every $n$, of either parity. We prove this in three steps. First, we have already seen that $M(3)<\frac35<1$. Next, we extend the map $n\mapsto M(n)$ to a real variable $y\in[2,\infty)$, which is already built into \eqref{eq:M_def}. We show that $M$ is a continuous and strictly convex function of $y\in[2,\infty)$, and that $M(y)\to\infty$ as $y\to\infty$ (Lemma \ref{lem:m_convex}). Together with $M(3)<1$, this implies that $M(y)=1$ for exactly one $y_*>3$, with $M<1$ on $[3,y_*)$ and $M>1$ on $(y_*,\infty)$. This already shows, without any numerical computation, that in odd dimension the equilateral box is a strict local minimum for $n<y_*$ and a strict local maximum for $n>y_*$. Finally, we use the computation in interval arithmetic of Section \ref{s:numerical} to locate $y_*$. 

We now carry out the second step. Convexity is not evident from \eqref{eq:M_def}. Hence we write $M$ as an integral in which $y$ appears only as the exponent of $\theta(t)-1$, against a positive weight $\kappa'$, where
\[\kappa(t)=-\frac{t \theta''(t)}{\theta'(t)} = - t \frac{d}{dt} \log(-\theta'(t)),\qquad t>0.\]
Since $\theta'<0$, $\kappa$ is $C^\infty$ on $(0,\infty)$. Moreover, replacing $t$ by $1/t$ in \eqref{eq:V_jacobi} gives $-\theta'(t)=\frac12t^{-3/2}\bigl(\theta(\tau)+2\tau\theta'(\tau)\bigr)$, with $\tau=1/t$. Using this expression we compute 
\begin{equation}\label{eq:kappa_jacobi}
\kappa(t)=\frac32+\frac{\tau\bigl(3\theta'(\tau)+2\tau\theta''(\tau)\bigr)}{\theta(\tau)+2\tau\theta'(\tau)},\qquad \tau=\frac1t.
\end{equation}

In the proof of the following lemma, we use the fact that $\kappa$ is strictly increasing, which we prove in Lemma \ref{lem:kappa} below.
\begin{lem}\label{lem:M_kappa}
For every $y\ge2$
\begin{equation}\label{eq:M_kappa}
M(y)=2\int_0^\infty\bigl(\theta(t)-1\bigr)^{y} \kappa'(t) dt.
\end{equation}
\end{lem}
\begin{proof}
We first describe $\theta+2t\theta'$ and $\kappa$ as $t\to\infty$ and as $t\to0^+$. As $t\to\infty$, $\theta(t)\to1$ and $t\theta'(t)\to0$, since $\theta'$ decays exponentially, hence $\theta(t)+2t\theta'(t)\to1$. By \eqref{eq:V_jacobi},
\[\theta(t)+2t\theta'(t)=-2t^{-3/2}\theta'(1/t)=4\pi t^{-3/2}\sum_{k\ge1}k^2e^{-\pi k^2/t},\]
hence $\theta+2t\theta'=\mathcal O\bigl(t^{-3/2}e^{-\pi/t}\bigr)$ as $t\to0^+$. In \eqref{eq:kappa_jacobi}, as $\tau\to\infty$ the numerator decays like $\tau^2e^{-\pi\tau}$, while the denominator tends to $1$. Thus $3-2\kappa(t)=\mathcal O\bigl(t^{-2}e^{-\pi/t}\bigr)$ as $t\to0^+$. In particular $\kappa(t)\to3/2$ as $t\to0^+$, and since $\kappa$ is strictly increasing by Lemma \ref{lem:kappa} below, $\kappa>3/2$ on $(0,\infty)$. 

Fix $0<\varepsilon<R$. Since $\frac{d}{dt}(\theta-1)^{y-1}=(y-1)(\theta-1)^{y-2}\theta'$, integrating by parts we get
\[y(y-1)\int_\varepsilon^R(\theta-1)^{y-2}\bigl(\theta+2t\theta'\bigr)(-\theta') dt=-y\Big[(\theta-1)^{y-1}\bigl(\theta+2t\theta'\bigr)\Big]_\varepsilon^R+y\int_\varepsilon^R(\theta-1)^{y-1}\bigl(\theta+2t\theta'\bigr)' dt.\]
By definition of $\kappa$, $\bigl(\theta+2t\theta'\bigr)'=3\theta'+2t\theta''=(2\kappa-3)(-\theta')$. Since $\frac{d}{dt}(\theta-1)^{y}=y(\theta-1)^{y-1}\theta'$ and $\kappa'=\frac{d}{dt}\bigl(\kappa-\frac32\bigr)$, integrating by parts again we get
\[y\int_\varepsilon^R(\theta-1)^{y-1}(2\kappa-3)(-\theta') dt=-\Big[(\theta-1)^y(2\kappa-3)\Big]_\varepsilon^R+2\int_\varepsilon^R(\theta-1)^y \kappa' dt.\]
All the integrands are positive, since $\theta+2t\theta'>0$ by \eqref{eq:V_jacobi}, $\kappa>3/2$ and $\kappa'>0$ by Lemma \ref{lem:kappa}.  Thus, as $\varepsilon\to0^+$ and $R\to\infty$, the integrals converge to the corresponding integrals on $(0,\infty)$, possibly equal to $+\infty$. It remains to show that the boundary terms vanish. As $t\to\infty$, $\theta(t)\to1$ and $\theta(t)+2t\theta'(t)\to1$, hence $(\theta-1)^{y-1}(\theta+2t\theta')\to0$. Moreover, since $\theta(t)-1=2\sum_{k\ge1}e^{-\pi k^2t}$ and $-\theta'(t)=2\pi\sum_{k\ge1}k^2e^{-\pi k^2t}$, we have $\theta(t)-1\le-\theta'(t)/\pi$ for every $t>0$, thus
\[0<(\theta-1)^y(2\kappa-3)<(\theta-1)^{y-1} \frac{(\theta-1) 2t\theta''}{-\theta'}\le(\theta-1)^{y-1} \frac{2t\theta''(t)}{\pi}\to0,\qquad t\to\infty.\]
As $t\to0^+$, the Jacobi identity gives $\theta(t)=t^{-1/2}\theta(1/t)\le\theta(1) t^{-1/2}$ for $t\le1$, so that $(\theta-1)^y$ and $(\theta-1)^{y-1}$ grow at most polynomially in $1/t$, while $\theta+2t\theta'=\mathcal O\bigl(t^{-3/2}e^{-\pi/t}\bigr)$ and $3-2\kappa=\mathcal O\bigl(t^{-2}e^{-\pi/t}\bigr)$. Hence both boundary terms vanish also as $t\to0^+$.
\end{proof}

It remains to prove that $\kappa$ is increasing. 

\begin{lem}\label{lem:kappa}
$\kappa$ is strictly increasing on $(0,+\infty)$, that is, $\kappa'(t)>0$ for every $t>0$. 
\end{lem}
\begin{proof}
In this proof we write $V(\tau)=\theta(\tau)+2\tau\theta'(\tau)=\sum_{k\in\mathbb Z}\big(1-2\pi k^2\tau\big)e^{-\pi k^2\tau}$, which is positive by \eqref{eq:V_jacobi}. We call $D_{\log}=\tau  d/d\tau$ the logarithmic derivative with respect to $\log\tau$, so that $\left(D_{\log}\log V\right)(\tau)=\tau V'(\tau)/V(\tau)$; moreover $D^2_{\log}=D_{\log}\circ{D_{\log}}$. Since $V'=3\theta'+2\tau\theta''$, \eqref{eq:kappa_jacobi} reads
\[\kappa(t)=\frac{3}{2}+\left(D_{\log}\log V \right)\left(\frac{1}{t}\right).\] 
We distinguish the cases $0<t\le1$ and $t\ge1$.

\emph{Case $0<t\le1$.} Let $\tau=1/t\ge1$. For any $C^1$ function $F$, 
\[\frac{d}{dt}\left[F(1/t)\right]=-\frac{1}{t^2}F'\left(\frac{1}{t}\right)=-\tau(D_{\log}F)(\tau)\]
so we have 
\[\kappa'(t)=-\tau(D^2_{\log}\log V)(\tau).\]
Hence, it is sufficient to show that $(D^2_{\log}\log V)(\tau)<0$ for $\tau\ge 1$. Since $V>0$,
\[D^2_{\log}\log V=\frac{D^2_{\log}V}{V}-\Big(\frac{D_{\log}V}{V}\Big)^2\le\frac{D^2_{\log}V}{V},\]
and we will show that $D^2_{\log}V(\tau)<0$. Define $\omega(x)=(1-2x)e^{-x}$ such that
\[V(\tau)=\sum_{k=-\infty}^{\infty}\omega(\pi k^2\tau).\]
Since $D_{\log}\big[\omega(c\tau)\big]=(D_{\log}\omega)(c\tau)$ for $c>0$, differentiating termwise and noting that the term $k=0$ is constant we get
\[D^2_{\log}V(\tau)=\sum_{k\neq0}(D^2_{\log}\omega)(\pi k^2\tau).\]
A direct computation gives
\[(D_{\log}\omega)(x)=x(2x-3)e^{-x},\qquad (D^2_{\log}\omega)(x)=-x(2x-1)(x-3)e^{-x},\]
so $D^2_{\log}\omega(x)<0$ for $x>3$. If $\tau\ge1$ and $k\neq0$ then $\pi k^2\tau\ge\pi>3$, hence every term of the series is negative and $D^2_{\log}V(\tau)<0$.

\emph{Case $t\ge1$.} Let $u=\pi t\ge\pi$, $\nu_l=l^2-1$ and
\[\eta(t)=\sum_{l\ge2}l^2e^{-\nu_lu}.\]
Then
\[-\theta'(t)=2\pi\sum_{k\ge1}k^2e^{-\pi k^2t}=2\pi e^{-u}\big(1+\eta(t)\big).\]
Since $-\theta'>0$, we have $\kappa=-t  \theta''/\theta'=-t  (\log(-\theta'))'=-D_{\log}\log(-\theta')$. From $\log(-\theta'(t))=\log(2\pi)-\pi t+\log(1+\eta(t))$ and $D_{\log}(\pi t)=\pi t$ we get
\[\kappa(t)=\pi t-\big(D_{\log}\log(1+\eta)\big)(t),\qquad\text{thus}\qquad t  \kappa'(t)=(D_{\log}\kappa)(t)=\pi t-\big(D^2_{\log}\log(1+\eta)\big)(t).\]
Moreover, $D^2_{\log}e^{-\nu_lu}=\nu_lu (\nu_lu-1) e^{-\nu_lu}$ and $\nu_lu\ge3\pi>1$, thus
\[0\le D^2_{\log}\eta=\sum_{l\ge2}l^2 \nu_lu (\nu_lu-1) e^{-\nu_lu}<u^2\sum_{l\ge2}l^2\nu_l^2 e^{-\nu_lu},\]
and as in the previous case
\[D^2_{\log}\log(1+\eta)=\frac{D^2_{\log}\eta}{1+\eta}-\Big(\frac{D_{\log}\eta}{1+\eta}\Big)^2\le D^2_{\log}\eta .\]
Hence
\[t \kappa'(t)\ge u-D^2_{\log}\eta>u (1-\Psi),\qquad \Psi=u\sum_{l\ge2}l^2\nu_l^2 e^{-\nu_lu},\]
and it remains to show that $\Psi<1$. We have $l^2=1+\nu_l\le e^{\nu_l}$ and $\nu_l^2\le e^{\nu_l}$ (from $e^{z-1}\ge z$ with $z=\nu_l/2$), while $\nu_l\ge3(l-1)$. Thus, by $e^z-1\ge z$,
\[\Psi\le u\sum_{l\ge2}e^{-\nu_l(u-2)}\le u\sum_{j\ge1}e^{-3(u-2)j}=\frac{u}{e^{3(u-2)}-1}\le\frac{u}{3(u-2)}<1,\]
where the last inequality holds because $u\ge\pi>3$. Hence $\kappa'(t)>0$ for $t\ge1$, which together with the first case shows that $\kappa$ is strictly increasing.
\end{proof}

\begin{lem}\label{lem:m_convex}
$M$ is finite, positive and continuous on $[2,\infty)$. $M$ is strictly convex on $[2,\infty)$, and $M(y)\to\infty$ as $y\to\infty$.
\end{lem}
\begin{proof}
$M$ is positive by \eqref{eq:M_def}, and $M(n)<\infty$ for every integer $n\ge2$ by Proposition \ref{prop:box_hess}. Consider an integer $a\ge2$ and $y\in[a,a+1]$. For any $s>0$ we have $s^y\le s^a+s^{a+1}$, thus, since $\theta(t) - 1>0$, and $\kappa'(t)>0$, we have 
\[ (\theta(t) - 1)^y \kappa'(t) \leq \left[ (\theta(t)-1)^a \kappa'(t) + (\theta(t)-1)^{a+1} \kappa'(t)\right], \quad \forall t>0. \]
Consequently, by \eqref{eq:M_kappa} and the dominated convergence theorem, $M$ is continuous on $y \in [a, a+1]$. 

For the convexity, for $s>0$, $s\neq1$, the function $y\mapsto s^y$ is strictly convex, thus
\[s^{\lambda y_1+(1-\lambda)y_2}<\lambda s^{y_1}+(1-\lambda) s^{y_2}.\]
We apply this inequality with $s=\theta(t)-1$, which is different from $1$ for every $t>0$ but one, since $\theta$ is strictly decreasing. Multiplying by $2\kappa'(t)>0$, integrating on $(0,\infty)$ and using \eqref{eq:M_kappa} shows that $M$ is strictly convex. 

For the limiting behavior as $y \to \infty$, by the Jacobi identity $\theta(t)=t^{-1/2}\theta(1/t)\ge t^{-1/2}$, we have $\theta(t)\to\infty$ as $t\to0^+$, and we can fix $t_0>0$ with $\theta(t_0)>2$. Since $\theta$ is decreasing, $\theta(t)-1\ge c =\theta(t_0)-1>1$ for $0<t\le t_0$. Since the integrand in \eqref{eq:M_kappa} is positive, for every $0<t_1<t_0$
\[M(y)\ge2c^y\int_{t_1}^{t_0}\kappa'(t) dt=2c^y\big(\kappa(t_0)-\kappa(t_1)\big),\]
which tends to $\infty$ as $y\to\infty$, since $c>1$ and $\kappa(t_0)>\kappa(t_1)$ by Lemma \ref{lem:kappa}. 
\end{proof}

\begin{proof}[Proof of Theorem \ref{th:local_extremum}]
As we already observed, for $n$ even $h_n>0$ by \eqref{eq:key}, since $M(n)>0$. Let now $n\ge3$ be odd. By \eqref{eq:Mn_one}, the sign of $h_n$ is the sign of $M(n)-1$. By Proposition \ref{prop:dir_three} and Lemma \ref{lem:m_convex}, $M(3)<1$, $M$ is continuous and convex, and $M(y)\to\infty$ as $y\to\infty$. Hence the set $\{y\ge3:\ M(y)\le1\}$ is convex, closed and bounded, that is, an interval $[3,y_*]$, and $M(y_*)=1$ by continuity. By convexity $M<1$ on $[3,y_*)$, while $M>1$ on $(y_*,\infty)$. Finally, by the computation in interval arithmetic of Section \ref{s:numerical} we obtain \eqref{eq:num_results}, in particular $M(23)<1<M(24)$, thus $23<y_*<24$. Hence $h_n<0$ for every odd $n$ with $3\le n\le23$, and $h_n>0$ for every odd $n\ge25$.
\end{proof}
In particular, since $y_*\in(23,24)$ is not an integer, and $h_n>0$ for every even $n$, the degenerate case $h_n=0$ does not occur in any dimension.

By Theorem \ref{thm_global_max}, a global maximum exists in high dimensions, and by Theorem \ref{th:local_extremum} the equilateral box is a strict local maximum there. It is therefore natural to conjecture the following.
\begin{conjecture}\label{conj:dirichlet}
For each $n\ge40$, among all $n$ dimensional box domains of unit volume, the zeta-regularized determinant of the Dirichlet Laplacian is uniquely maximized by the equilateral box.
\end{conjecture}
\section{Interval arithmetic computation of the critical odd dimension}
\label{s:numerical}
In this section we prove \eqref{eq:num_results}, which locates the odd dimension at which the Hessian of the spectral height at the equilateral Dirichlet box changes from negative to positive definite. By the proof of Theorem \ref{th:local_extremum}, for odd $n$ the equilateral box is a strict local minimum if $n<y_*$ and a strict local maximum if $n>y_*$, where $y_*$ is the unique $y>3$ with $M(y)=1$. It therefore suffices to show that $M(23)<1<M(24)$, which gives $y_*\in(23,24)$. We also compute $M(25)$ and, as a check, $M(3)$.

Interval arithmetic computes a closed interval that is guaranteed to contain the exact value of a quantity. Arithmetic operations and elementary functions are evaluated on intervals in such a way that the result contains all possible values, and the endpoints are rounded outwards; hence, if the inputs lie in the input intervals, the exact result lies in the output interval. We use the ball arithmetic of the library Arb, now part of FLINT \cite{Arb,FLINT}, in which intervals are represented as balls $[m\pm r]$. For the computation we use neither \eqref{eq:M_def} nor \eqref{eq:M_kappa} directly, but
\begin{equation}\label{eq:int_art2}
    M(y)=y(y-1)\int_1^\infty \bigl(-\theta'(t)\bigr)\bigl(\theta(t)+2t\theta'(t)\bigr)\Bigl[\bigl(2\ttheta(t)\bigr)^{y-2}+t\bigl(\sqrt{t} \theta(t)-1\bigr)^{y-2}\Bigr]dt.
\end{equation}
We obtain \eqref{eq:int_art2} by splitting the integral in \eqref{eq:M_def} at $t=1$. On $[1,\infty)$ we use $\theta-1=2\ttheta$. On $(0,1]$ we substitute $t=1/s$ and use the Jacobi identity $\theta(1/s)=\sqrt s \theta(s)$ together with \eqref{eq:V_jacobi}, which give $\theta(1/s)+\frac2s\theta'(1/s)=-2s^{3/2}\theta'(s)$ and $-\theta'(1/s)=\frac12s^{3/2}\bigl(\theta(s)+2s\theta'(s)\bigr)$. The advantage of \eqref{eq:int_art2} is numerical: close to $t=0$ the theta series converge slowly and $\theta+2t\theta'$ is the difference of two large and almost equal quantities, whereas for $t\ge1$ the terms of the series decay very fast. We denote by $f_y$ the integrand in \eqref{eq:int_art2}, including the factor $y(y-1)$. For $y\ge2$ and $t\ge1$ every factor of $f_y$ is non negative, by \eqref{eq:V_jacobi} and since $\sqrt t \theta(t)=\theta(1/t)\ge1$; hence $f_y\ge0$ on $[1,\infty)$. We split $[1,\infty)$ into a compact interval $[1,a]$ and a tail $[a,\infty)$. On $[1,a]$ we compute an interval containing $\int_1^af_y$, and on $[a,\infty)$ we bound the integral from above by an explicit quantity $\mathcal R_a(y)$, so that the contribution of the tail lies in $[0,\mathcal R_a(y)]$. If the upper endpoint of the interval obtained on $[1,a]$, plus $\mathcal R_a(y)$, is less than $1$, we conclude that $M(y)<1$; if its lower endpoint is larger than $1$, we conclude that $M(y)>1$. The integrand only involves $\theta=1+2\ttheta$, $\ttheta(t)=\sum_{k\ge1}w^{k^2}$ and $-\theta'(t)=2\pi\sum_{k\ge1}k^2w^{k^2}$, where $w=e^{-\pi t}$. We truncate both series at $k=k_0$. For $k>k_0$ the exponents $k^2$ are distinct integers with $k^2\ge j_0=(k_0+1)^2$; hence, if $0<w\le\rho<1$,
\begin{align}
    0\le\sum_{k>k_0}w^{k^2}&\le \sum_{j\ge j_0}\rho^{j}=\frac{\rho^{j_0}}{1-\rho},\label{eq:rem0}\\
    0\le\sum_{k>k_0}k^2w^{k^2}&\le \sum_{j\ge j_0}j\rho^{j}=\frac{\rho^{j_0}\bigl(j_0-(j_0-1)\rho\bigr)}{(1-\rho)^2}\le\frac{j_0 \rho^{j_0}}{(1-\rho)^2}.\label{eq:rem1}
\end{align}
In the computation, the truncated sums are evaluated in interval arithmetic, and the resulting intervals are enlarged by the bounds \eqref{eq:rem0} and \eqref{eq:rem1}. We now bound the contribution of the tail $[a,\infty)$, for $a>\max\{1,\frac{y}{2\pi}\}$. Let $t\ge a$, so that $0<w=e^{-\pi t}\le e^{-\pi}$, and let $c_{\ttheta}=(1-e^{-\pi})^{-1}$. The bounds \eqref{eq:rem0} and \eqref{eq:rem1} with $k_0=0$, that is $j_0=1$, and $\rho=w$ give
\[
\ttheta(t)\le\frac{w}{1-w}\le c_{\ttheta} w,
\qquad
-\theta'(t)\le\frac{2\pi w}{(1-w)^2}\le2\pi c_{\ttheta}^2 w .
\]
Moreover,
\[
\theta(t)+2t\theta'(t)=1+2\sum_{k\ge1}\bigl(1-2\pi k^2t\bigr)w^{k^2}<1,
\]
because $2\pi k^2t\ge2\pi>1$ for every $k\ge1$. Finally, $\theta(t)=1+2\ttheta(t)\le1+2c_{\ttheta}e^{-\pi a}$, hence $0\le\sqrt t \theta(t)-1\le\sqrt t \bigl(1+2c_{\ttheta}e^{-\pi a}\bigr)$. Therefore, on $[a,\infty)$,
\[
0\le f_y(t)\le2\pi c_{\ttheta}^2 y(y-1)\Bigl[(2c_{\ttheta})^{y-2}e^{-\pi(y-1)t}
+\bigl(1+2c_{\ttheta}e^{-\pi a}\bigr)^{y-2} t^{y/2}e^{-\pi t}\Bigr].
\]
The first term can be integrated explicitly,
\[\int_a^\infty e^{-\pi(y-1)t} dt=\frac{e^{-\pi(y-1)a}}{\pi(y-1)}.\]
For the second one, $\frac y2\log t-\pi t$ is concave in $t$, so it lies below its tangent at $t=a$, whose slope $\frac{y}{2a}-\pi$ is negative by the choice of $a$. Therefore
\[
\int_a^\infty t^{y/2}e^{-\pi t} dt\le a^{y/2}e^{-\pi a}\int_a^\infty e^{-(\pi-\frac{y}{2a})(t-a)} dt
=\frac{a^{y/2}e^{-\pi a}}{\pi-\frac{y}{2a}},
\]
and the contribution of the tail to $M(y)$ lies in $[0,\mathcal R_a(y)]$, where
\begin{equation}\label{eq:tail}
\mathcal R_a(y)=2\pi c_{\ttheta}^2 y(y-1)\left[\frac{(2c_{\ttheta})^{y-2} e^{-\pi(y-1)a}}{\pi(y-1)}
+\bigl(1+2c_{\ttheta}e^{-\pi a}\bigr)^{y-2} \frac{a^{y/2}e^{-\pi a}}{\pi-\frac{y}{2a}}\right].
\end{equation}

Next we integrate on $[1,a]$. We divide $[1,a]$ into $N$ subintervals $[t_i,t_{i+1}]$, where $t_i=1+ih$ for $i=0,\dots,N$ and $h=(a-1)/N$. If $Y_i$ is an interval such that $f_y\bigl([t_i,t_{i+1}]\bigr)\subseteq Y_i$, then by the mean value theorem
\[\int_{t_i}^{t_{i+1}}f_y=h f_y(\xi_i)\in hY_i \quad \text{for some } \xi_i\in[t_i,t_{i+1}], \qquad\text{and thus}\qquad \int_1^af_y\in h\sum_{i=0}^{N-1} Y_i,\]
where the sum of the intervals is computed in interval arithmetic as well. We compute $Y_i$ by evaluating $f_y$ in interval arithmetic with input $[t_i,t_{i+1}]$; on each subinterval the remainders of the theta series are bounded by \eqref{eq:rem0} and \eqref{eq:rem1}, with $\rho$ an upper bound of $e^{-\pi t_i}$.

We ran the computation with $a=20$, $h=2^{-10}$ (that is, $N=19\cdot2^{10}=19456$ subintervals), $k_0=6$ and $128$ bits of working precision, using FLINT 3.0.1. Since $h$ is a power of $2$, the endpoints $t_i$ and the multiplication by $h$ are exact. For $y\le25$ the tail bound \eqref{eq:tail} is smaller than $2\cdot10^{-8}$. Rounding the resulting intervals outwards, we obtain
\begin{equation}\label{eq:num_results}
M(23)\in[0.9085, 0.9142],\quad M(24)\in[1.4807, 1.4899],\quad M(25)\in[2.4690, 2.4843].
\end{equation}
The same computation gives
\[M(3)\in[0.1344, 0.1355],\]
in agreement with Proposition \ref{prop:dir_three}, which gives $M(3)<\frac35$. The whole computation takes less than a second on a laptop, and the code is available at \url{https://github.com/fabiofr-cth/odd-dim-dirichlet-boxes-threshold}.
In particular, \eqref{eq:num_results} gives $M(23)<1<M(24)$, which is the input needed in the proof of Theorem \ref{th:local_extremum}.

\section{The behavior of the determinant as the dimension increases}
\label{s:dimension} 
Here, we show that, for orthogonal tori, Neumann boxes and Dirichlet boxes alike, normalized to have unit volume, the supremum of the zeta-regularized determinant tends to zero.

\begin{proof}[Proof of Theorem \ref{thm:dim_infty}]
We begin with orthogonal tori.  The spectral height of the standard torus $\R^n /\Z^n$ is given by a constant plus
\beq  \mathcal I (n) = - \frac 2 n - \gamma-\log\pi + \sum_{a \in \mathbb{Z}^n}'  \int_1^\infty e^{-\pi ||a||^2 t}\left(  t^{-1}
+ t^{ n/2 - 1}\ \right) dt \label{eq:awesome_ending} 
\eeq 
\[ = - \frac 2 n - \gamma-\log\pi+  \sum_{a \in \Z^n}' \left[ E_1 (\pi ||a||^2) +(\pi ||a||^2)^{-n/2} \Gamma\left( \frac n 2, \pi ||a||^2 \right) \right]. \] 
Here, $E_1(c) = \int_1 ^\infty e^{-ct} t^{-1} dt$, and $\Gamma(s, x)$ is the incomplete upper Gamma function.  The first term, $-\frac 2 n$ is an increasing function of $n$ and tends to $0$ as $n \to \infty$.  The terms $-\gamma - \log \pi$ are independent of $n$. For each fixed $a \in \Z^n$, the term 
\[ (\pi ||a||^2)^{-n/2} \Gamma\left( \frac n 2, \pi ||a||^2 \right) = \int_1 ^\infty t^{n/2-1} e^{-\pi ||a||^2t} dt  \]
is an increasing function of $n$. To be more precise, we observe that $\Z^n$ is canonically embedded in $\Z^{n+1}$, corresponding to the set 
\[ \Z^{n+1} _0 = \{z = (z_1, \ldots, z_n, 0) \in \Z^{n+1}\}.\]   
For each $a \in \Z^{n+1} _0$, $||a||^2$ is identical to the norm of the corresponding $a \in \Z^n$.  Thus for such $a$ we have 
\[ (\pi ||a||^2)^{-(n+1)/2} \Gamma\left( \frac {n+1}{2} , \pi ||a||^2 \right) \geq (\pi ||a||^2)^{-n/2} \Gamma\left( \frac n 2, \pi ||a||^2 \right). \]
Write 
\[ H(a, n) = E_1 (\pi ||a||^2) +(\pi ||a||^2)^{-n/2} \Gamma\left( \frac n 2, \pi ||a||^2 \right). \]
Then, for $a \in \Z^{n+1} _0$,  
\[ H(a, n+1) \geq H(a, n).\]
Consequently, 
\beq \sum'_{a \in \Z^{n+1}} H(a, n+1) &=& \sum_{a \in \Z^{n+1} _0} ' H(a, n+1) + \sum_{a \in \Z^{n+1} \setminus \Z^{n+1} _0} ' H(a, n+1) \nn \\ 
&=& \sum_{a \in \Z^n} ' H(a, n+1) +  \sum_{a \in \Z^{n+1} \setminus  \Z^{n+1} _0} ' H(a, n+1) \nn \\ 
& > & \sum_{a \in \Z^n} ' H(a, n) + E_1 (\pi) \nn \eeq 

In the last step, we used the single term in the sum corresponding to the unit vector $e_{n+1}$ in $\Z^{n+1}$.  We note that this term is independent of $n$.  Since the term $-\frac 2 n$ is increasing with $n$, and the other terms $-\gamma - \log \pi$ are constant, this shows that 
\[  \mathcal I (n+1) >  \mathcal I (n) + E_1 (\pi) \implies \mathcal I(n+2) >  \mathcal I (n+1) + E_1 (\pi) >  \mathcal I (n) + 2 E_1 (\pi),\]    
and more generally 
\[ \mathcal I(2+k) \geq \mathcal I(2) + k E_1(\pi) \to \infty \textrm{ as } k \to \infty.\]
Consequently, the spectral height of the unit equilateral orthogonal torus of dimension $n$ is a strictly increasing function of the dimension $n$, and in contrast, the determinant is a strictly decreasing function of the dimension.  Moreover, $ \mathcal I (n) \to \infty$ as $n \to \infty$, which in turn shows that the determinant tends to zero as the dimension tends to infinity.

For Neumann boxes, by Theorem \ref{th:awesome_neumann} it suffices to consider the cube. By \eqref{eq:zetacirc_zetae} and \eqref{eq:neumann0}, the Neumann spectral height of the unit equilateral $n$-dimensional box is 
\[ \zeta_\scN'(0, I_n) = (2^{-n+1} - 2) \log 2  + 2^{-n} \sum_{k=1} ^n {\binom{n}{k}} \zeta_\circ '(0, I_k). \]
Since the terms in the sum are non negative and increasing, and $\sum_{k\ge n/2}\binom nk\ge2^{n-1}$, we estimate 
\[\zeta_\scN'(0, I_n) \geq (2^{-n+1} - 2) \log 2 + \frac 1 2 \zeta_\circ '(0, I_k), \quad k = \left \lceil \frac n 2 \right \rceil . \]
The first term tends to $-2 \log 2$, while the term $\frac 1 2 \zeta_\circ '(0, I_k)$ tends to infinity as $n \to \infty$.  Consequently, the Neumann spectral height of all $n$-dimensional unit volume Euclidean boxes tends to infinity as the dimension $n \to \infty$, and thus the determinant tends to zero. 

For Dirichlet boxes we cannot reduce to a single box, and we argue directly on the whole of $\mathcal S_n$ through Lemma \ref{lem:lower}, which gives the lower bound 
\[ \zeta_\scD'(0, Q) \geq \left( - \frac 1 2 \right)^n (\gamma - \log \pi + 2 \log q) + 2^{1-n} q^{-n} \left( \mathcal G_+ (y) - \mathcal G_- (y) \right).\]
Here, $q$ is the shortest side length of the box associated with the matrix $Q$, and $y=(y_1, \ldots, y_n)$ with $y_i = q/q_i$, with $q_i$ the side lengths of the box. Then $q \leq 1$.  For $u \geq 1$, every factor $u-y_i$ in the definition of $\mathcal G_+(y)$ is greater than or equal to $u-1 \geq 0$.  Thus for dimensions greater than or equal to $2$, 
\[
\mathcal G_+(y)\ \ge\ 2\int_1^\infty(u-1)^{n-1}e^{-\pi u^2}du\ \ge\ \mathcal K 3^{n-1},\qquad \mathcal K=2\int_4^\infty e^{-\pi u^2}du .
\]
Moreover $\mathcal G_-(y)\le n$ by Lemma \ref{lem:vertex}, so we have 
\[
\zeta'_\scD(0,Q)\ \ge\ -2^{-n}|\gamma-\log\pi|+2^{1-n}\Bigl(q^{-n}\bigl(\mathcal K 3^{n-1}-n\bigr)-\log\frac1q\Bigr).
\]
Let $n$ be sufficiently large such that $\mathcal K 3^{n-1}\ge n+1$.  Then we have 
\[q^{-n} = e^{-n \log q} \geq 1 + n \log \frac 1 q \geq \log \frac 1 q \implies n q^{-n} + \log \frac 1 q \leq (n+1) q^{-n}.\] 
Consequently, 
\[ q^{-n}(\mathcal K 3^{n-1} -n) - \log(1/q) \geq q^{-n}(\mathcal K 3^{n-1} - n - 1) \geq \mathcal K 3^{n-1} - n - 1.\]
We then have the estimate 
\[
\zeta'_\scD(0,Q)\ \ge\ \mathcal K\Bigl(\frac32\Bigr)^{n-1}-2^{1-n}
(n+1)-2^{-n}|\gamma-\log\pi|\ \longrightarrow\ +\infty, \textrm{ as } n \to \infty.
\]
Consequently, the spectral height of Dirichlet boxes of unit volume also tends to infinity as the dimension tends to infinity, and as a result, the determinant tends to zero. 
\end{proof}

The decay of the determinant as the dimension tends to infinity is related to the phenomenon we observed in the Dirichlet case.  In lower and intermediate dimensions, the Dirichlet determinant is unbounded, while in all dimensions 40 and higher, the determinant has a global maximum.  In our analysis, we observed that this is due to the fact that as the dimension grows, the term which dominates the Mellin transform expression for the spectral height changes.  In lower dimensions, it is the first part of the short time asymptotic expansion of the heat trace, the terms $c_n t^{-n/2} + \ldots + c_0$, while in high dimensions, it is the remainder term that dominates.  This remainder is the term that decays as $\mathcal O (e^{-c/t})$ as $t \to 0$ for a constant $c>0$ determined by the length of the shortest closed geodesics.  As the dimension increases, so does the number of closed geodesics of a given length.  This is true in all three cases; increasing dimension amplifies the contributions of paths winding around periodic directions or reflecting between opposite faces.  

From a physical perspective, in terms of a Gaussian scalar field, its inverse square root is the zeta-regularized partition function (omitting zero modes when present, as in the periodic and Neumann cases).  The limiting behavior as the dimension increases corresponds to an increasingly negative renormalized fluctuation free energy.  Thus, the fact that the spectral height tends to $+\infty$ as the dimension tends to $\infty$ reflects the fact that the field has more and more opportunity to fluctuate as the dimension grows, similar to the fact that there are more and more short closed geodesics. This provides a physical interpretation of the collective spectral effect captured by our estimates: fluctuations contribute increasingly strongly after the local short-time divergences have been removed by regularization.

 \section{Spectral determinant and Casimir energy}
\label{s:conclusion}
Empty space is never really empty: every mode of a quantum field keeps oscillating in its ground state, with zero point energy $\frac12\sqrt{\lambda_k}$. The sum of these energies diverges, but zeta regularization leaves a finite, shape dependent part, the Casimir energy 
\[E_{\mathrm{Cas}}(\mathcal B)=\frac12\zeta_\scD\bigl(-\frac12,Q\bigr)\] of a massless scalar field in the box $\mathcal B$, in units $\hbar=c=1$. Equation \eqref{eq:dir_mess} has a simple physical reading: in a long box the long side plays the role of Euclidean time, and $\log\det\Delta$ is, to leading order, twice its length times the Casimir energy of the cross-section. A positive vacuum energy makes the determinant blow up, so our results can be read backwards, as statements about the vacuum. Since $E_{\mathrm{Cas}}$ scales like the inverse of a length, its sign depends only on the shape of the box, and our proofs give a fairly complete map of it. Corollary \ref{cor:casimir_2boxes} and Proposition \ref{prop:casimir_pos_odd} essentially state that the hypercube has negative energy in every odd dimension and positive energy in every even dimension $2\le m\le28$. Shape matters as much as dimension: in odd dimension $3\le m\le21$, long thin needles with an even-dimensional cubic cross-section have positive energy by Proposition \ref{prop:cas_pos_even}, although the cube does not. Thus boxes with positive Casimir energy exist for $2\le m\le22$ and for $m=24,26,28$. Conversely, no box of dimension $m\ge39$ has positive Casimir energy: it would make the determinant unbounded in dimension $m+1\ge40$, against Theorem \ref{thm_global_max}. It is remarkable that a statement about all boxes at once follows from the existence of a maximum of the determinant, without computing a single Casimir energy. The dimensions $m=23,25,27$ and $29\le m\le38$ remain open.

Behind these thresholds there is a competition that the heat trace makes visible. By Lemma \ref{lem:zeta_half}, the Casimir energy only sees the exponentially small remainder $E_Q$ of the heat trace, made of the paths that bounce back and forth between opposite faces of the box. For the hypercube, a positive contribution, present only in even dimension, decays like $1/n$, while the negative contribution of the bouncing paths grows with the number of directions in which they can bounce; the same mechanism drives the determinant to zero in Section \ref{s:dimension}. Roughly speaking, the sign of the Casimir energy tells whether the vacuum pushes the walls of the box outwards or pulls them inwards: in high dimension, it never pushes. The Neumann case gives a useful comparison. The computation leading to \eqref{eq:dir_mess} applies to the Neumann Laplacian as well, with the same leading term $-\zeta_\scN(-\frac12,Q_0)L^{\frac{n}{n-1}}$ and remaining terms of order $\log L$. Since by Theorem \ref{th:awesome_neumann} the Neumann determinant is bounded on $\mathcal S_n$ for every $n$, no box of any dimension has positive Neumann Casimir energy $\frac12\zeta_\scN(-\frac12,Q)$. Here, however, the determinant tells us nothing new. In Proposition \ref{prop:box} all the coefficients of the Neumann zeta function are positive, and each Epstein zeta function is negative at $s=-\frac12$: by the functional equation in Section \ref{s:2}, after rescaling to unit volume, it has the sign of $\Gamma(-\frac12)<0$. Hence the Neumann Casimir energy of every box is negative, and with Neumann conditions the vacuum always pulls the walls inwards. In the Dirichlet case the signs in Proposition \ref{prop:box} alternate, and the competition between sub-boxes of different dimensions produces the dimension dependent behavior described above. The functional equation of the Epstein zeta function also shows that every flat torus has negative Casimir energy, so for tori the question has a trivial answer.

Boxes are among the oldest objects of spectral theory, yet their determinant still hides thresholds whose deeper origin we do not fully understand: between $23$ and $25$ for the local behavior, between $28$ and $30$ for the Casimir energy of the hypercube, and a global maximum that appears only in high dimension. Special dimensions are no strangers to zeta-regularized vacuum energies: in bosonic string theory, the value $\zeta_R(-1)=-\frac1{12}$ is compatible with Lorentz invariance only in spacetime dimension $26$. Whether our thresholds have a deeper meaning remains to be seen. Even an empty box, it seems, still has a lot to say.

\appendix

\section{Combinatorial Proofs} \label{app:auxiliary}
Here we collect the proofs of the combinatorial identities relating 
the Dirichlet, Neumann, and orthogonal torus determinants. 
\begin{proof}[Proof of Proposition \ref{prop:box}]
To prove the Dirichlet case consider, for $\Re s>\frac 
n2$,
\[
\zeta_{\scD}(s,Q)=\pi^{-2s}\sum_{k\in\mathbb{N}_+^n}Q^{-2}[k]^{-s},
\qquad
\zeta_E(s,Q^{-2})=\sum'_{k\in\mathbb{Z}^n\setminus\{0\}}Q^{-2}[k]^{-
s}.
\]
In this half-plane all the series below converge absolutely, and the identities we obtain extend to every $s\in\mathbb C$ by meromorphic continuation. We define the sets
\[
\mathcal{S} =\{k\in\mathbb{Z}^n\mid k_i\neq0\ \forall i=1,\dots,n\}
\qquad\text{and}\qquad
\mathcal{Z} =\mathbb{Z}^n\setminus\{0\}.
\]
Since $Q^{-2}[k]$ depends only on the squares $k_i^2$, it is invariant under any reflection $k_i\mapsto-k_i$. Since there are $2^n$ sign combinations for a vector with no zero coordinates, we have
\[
\sum_{k\in\mathbb{N}_+^n}Q^{-2}[k]^{-s}=2^{-
n}\sum_{k\in\mathcal{S}}Q^{-2}[k]^{-s}.
\]
We now express $\mathcal{S}$ via inclusion exclusion. For 
$i=1,\dots,n$ define the boundary terms of $\mathcal{Z}$
\[
\mathcal{A}_i =\{k\in\mathcal{Z}\mid k_i=0\}.
\]
Then $\mathcal{S}$ is the complement of their union, 
$\mathcal{S}=\mathcal{Z}\setminus\bigcup_{i=1}^n\mathcal{A}_i$, and 
on $\mathcal{Z}$ we can express the indicator function of 
$\mathcal{S}$ as
\[
\mathbb{I}_{\mathcal{S}}=\prod_{i=1}^n\bigl(1-
\mathbb{I}_{\mathcal{A}_i}\bigr)=\sum_{L\subseteq\{1,\dots,n\}}
(-1)^{|L|} \mathbb{I}_{\mathcal{A}_L},
\qquad\text{where}\quad
\mathcal{A}_L =\bigcap_{i\in L}\mathcal{A}_i=\{k\in\mathcal{Z}\mid k_i=0\ \forall i\in L\}.
\]
Multiplying by $Q^{-2}[k]^{-s}$ and summing over all $k\in\mathcal{Z}$, we obtain
\[
\sum_{k\in\mathcal{S}}Q^{-2}[k]^{-s}=\sum_{L\subseteq\{1,\dots,n\}}(-1)^{|L|}\sum_{k\in\mathcal{A}_L}Q^{-2}[k]^{-s}.
\]
Let $J =L^c$ be the complement of $L$ within $\{1,\dots,n\}$. The condition $k\in\mathcal{A}_L$ forces all coordinates indexed by $L$ to vanish, so the sum over $\mathcal{A}_L$ reduces to a sum over the non zero vectors of coordinates indexed by $J$:
\[
\sum_{k\in\mathcal{A}_L}Q^{-2}[k]^{-s}=\zeta_E(s,Q_J^{-2}).
\]
This holds also for $J=\emptyset$, since $\mathcal{A}_{\{1,\dots,n\}}=\emptyset$ and $\zeta_E(s,Q_\emptyset^{-2})=0$. Since $|L|=n-|J|$, substituting back into the sum gives
\begin{align*}
\zeta_{\scD}(s,Q)
&=\pi^{-2s}\left[2^{-n}\sum_{J\subseteq\{1,\dots,n\}}(-1)^{n-|J|} \zeta_E(s,Q_J^{-2})\right]\\
&=\pi^{-2s}\left[2^{-n}\sum_{m=0}^n(-1)^{n-m}\sum_{\substack{J\subseteq\{1,\dots,n\}\\|J|=m}}\zeta_E(s,Q_J^{-2})\right].
\end{align*}
The expression with $\zeta_\circ$ follows from $\zeta_\circ(s,Q_J)=(2\pi)^{-2s}\zeta_E(s,Q_J^{-2})$, since the eigenvalues of the torus $\mathbb R^{|J|}/Q_J\mathbb Z^{|J|}$ are $4\pi^2Q_J^{-2}[k]$, $k\in\mathbb Z^{|J|}$. Evaluating at $s=0$ using \eqref{eq:zetaE(0)}, every non empty $J$ contributes $-1$, and there are $\binom nm$ subsets $J$ with $|J|=m$. Since $(-1)^{n-m}=(-1)^n(-1)^m$, we have
\[
\zeta_{\scD}(0,Q)=-2^{-n}\sum_{m=1}^n(-1)^{n-m}\binom nm=-\left(-\frac12\right)^n\sum_{m=1}^n(-1)^m\binom nm=\left(-\frac12\right)^n,
\]
where we have used the standard identity
\begin{equation}\label{eq:standardid}
\sum_{m=1}^n(-1)^m\binom nm=(1-1)^n-1=-1.
\end{equation}
Taking the derivative and evaluating at $s=0$, since $\pi^{-2s}=1-2s\log\pi+\mathcal O(s^2)$, we get
\begin{equation}\label{eq3}
\begin{split}
\zeta_{\scD}'(0,Q)={}&-2\log\pi\left[2^{-n}\sum_{m=0}^n(-1)^{n-m}\sum_{\substack{J\subseteq\{1,\dots,n\}\\|J|=m}}\zeta_E(0,Q_J^{-2})\right]\\
&+2^{-n}\sum_{m=0}^n(-1)^{n-m}\sum_{\substack{J\subseteq\{1,\dots,n\}\\|J|=m}}\zeta'_E(0,Q_J^{-2}).
\end{split}
\end{equation}
The term in square brackets is $\zeta_\scD(0,Q)=(-\frac12)^n$, computed above. This proves \eqref{eq5}. To prove the Neumann case we start from
\[
\zeta_{\scN}(s,Q)=\pi^{-2s}\sum_{k\in\mathbb{N}^n\setminus\{0\}}Q^{-2}[k]^{-s},
\]
and we decompose the summation set into the disjoint union
\[
\mathbb{N}^n\setminus\{0\}=\bigcup_{\emptyset\neq J\subseteq\{1,\dots,n\}}\{k\in\mathbb{N}^n\mid k_i>0\text{ if }i\in J\text{ and }k_i=0\text{ if }i\notin J\}.
\]
Then
\[
\sum_{k\in\mathbb{N}^n\setminus\{0\}}Q^{-2}[k]^{-s}=\sum_{J\subseteq\{1,\dots,n\}}\left(\sum_{k\in\mathbb{N}_+^{|J|}}Q_J^{-2}[k]^{-s}\right),
\]
where the term with $J=\emptyset$ is zero. From the Dirichlet case, applied to the box $\prod_{i\in J}[0,q_i]$, we know that
\[
\sum_{k\in\mathbb{N}_+^{|J|}}Q_J^{-2}[k]^{-s}=2^{-|J|}\sum_{K\subseteq J}(-1)^{|J|-|K|}\zeta_E(s,Q_K^{-2}).
\]
Substituting, we find
\[
\zeta_{\scN}(s,Q)=\pi^{-2s}\sum_{J\subseteq\{1,\dots,n\}}2^{-|J|}\sum_{K\subseteq J}(-1)^{|J|-|K|}\zeta_E(s,Q_K^{-2}).
\]
We can exchange the summation order: we first fix $K$, and then we sum over all subsets $J$ such that $K\subseteq J\subseteq\{1,\dots,n\}$. In this way we isolate the coefficient of each zeta function:
\[
\zeta_{\scN}(s,Q)=\pi^{-2s}\sum_{K\subseteq\{1,\dots,n\}}\zeta_E(s,Q_K^{-2})\left(\sum_{K\subseteq J\subseteq\{1,\dots,n\}}2^{-|J|}(-1)^{|J|-|K|}\right).
\]
Finally we compute the coefficient for a given $K$. Let $|K|=h$ and $r=|J|-h$. For a fixed $K$, there are $\binom{n-h}{r}$ choices of $J$ with $|J|=h+r$, hence
\[
\sum_{K\subseteq J\subseteq\{1,\dots,n\}}2^{-|J|}(-1)^{|J|-|K|}
=2^{-h}\sum_{r=0}^{n-h}\binom{n-h}{r}\left(-\frac{1}{2}\right)^r=2^{-h}\left(\frac{1}{2}\right)^{n-h}=2^{-n}.
\]
Therefore the coefficient is constant and independent of $K$, and
\[
\zeta_{\scN}(s,Q)=\pi^{-2s}\left[2^{-n}\sum_{K\subseteq\{1,\dots,n\}}\zeta_E(s,Q_K^{-2})\right].
\]
Since $\zeta_E(0,Q_\emptyset^{-2})=0$ and $\zeta_E(0,Q_K^{-2})=-1$ for $K\neq\emptyset$ by \eqref{eq:zetaE(0)}, we compute
\[
\zeta_\scN(0,Q)=2^{-n}\sum_{\emptyset\neq K\subseteq\{1,\dots,n\}}(-1)=2^{-n}(1-2^n)=\left(\frac{1}{2}\right)^n-1.
\]
Taking the derivative at zero as in the Dirichlet case,
\[
\zeta'_{\scN}(0,Q)=-2\log\pi \zeta_\scN(0,Q)+2^{-n}\sum_{K\subseteq\{1,\dots,n\}}\zeta'_E(0,Q_K^{-2}),
\]
which is \eqref{eq:neumann0}.
\end{proof}
\begin{proof}[Proof of Corollary \ref{cor:box_unit}]
Let $\emptyset\neq J\subseteq\{1,\dots,n\}$. Since $\tilde Q_J^{-2}[k]=c_J^{-2} Q_J^{-2}[k]$, we have
\[
\zeta_E(s,Q_J^{-2})=c_J^{-2s} \zeta_E(s,\tilde Q_J^{-2}).
\]
Differentiating at $s=0$ and using \eqref{eq:zetaE(0)},
\[
\zeta'_E(0,Q_J^{-2})=-2\log(c_J) \zeta_E(0,\tilde Q_J^{-2})+\zeta'_E(0,\tilde Q_J^{-2})=2\log(c_J)+\zeta'_E(0,\tilde Q_J^{-2}).
\]
We substitute this into \eqref{eq5} and \eqref{eq:neumann0}, where the terms with $J=\emptyset$ vanish. The new terms are
\[
2^{1-n}(-1)^n\sum_{\emptyset\neq J}(-1)^{|J|}\log(c_J),\qquad
2^{1-n}\sum_{\emptyset\neq J}\log(c_J),
\]
for Dirichlet and Neumann respectively.
Hence, it suffices to show that $\sum_{\emptyset\neq J}\varepsilon^{|J|}\log(c_J)=0$ for $\varepsilon=\pm1$. Since $\log(c_J)=-\frac{1}{|J|}\sum_{i\in J}\log(q_i)$, exchanging the two sums gives
\[
\sum_{\emptyset\neq J}\varepsilon^{|J|}\log(c_J)=-\sum_{i=1}^n\log(q_i)\sum_{J\ni i}\frac{\varepsilon^{|J|}}{|J|}.
\]
For a fixed $i$, there are $\binom{n-1}{k-1}$ subsets $J\ni i$ with $|J|=k$, and $\frac1k\binom{n-1}{k-1}=\frac1n\binom nk$. Thus
\[
\sum_{J\ni i}\frac{\varepsilon^{|J|}}{|J|}=\frac1n\sum_{k=1}^n\binom nk\varepsilon^k=\frac{(1+\varepsilon)^n-1}{n},
\]
which does not depend on $i$. Hence
\[
\sum_{\emptyset\neq J}\varepsilon^{|J|}\log(c_J)=-\frac{(1+\varepsilon)^n-1}{n} \log\Bigl(\prod_{i=1}^nq_i\Bigr)=-\frac{(1+\varepsilon)^n-1}{n} \log(\det Q)=0.\qedhere
\]
\end{proof}

\section{The cubic torus uniquely maximizes the determinant} \label{app:no_montgomery}

Although one can use the results of \cite{Montgomery1988} to show that among all orthogonal tori of unit volume, $\R^n /\Z^n$ uniquely maximizes the determinant, we first proved this fact independently without \cite{Montgomery1988}.  Since the proof has its own ideas, which may be useful for other problems, and it is relatively short, we present it here. 

\begin{proof}[Proof of Corollary \ref{th:awesome_torus}]
We aim to show that the spectral height of an orthogonal flat torus of unit volume is uniquely minimized at $q_1=\cdots=q_n=1$. We note that no generality is lost by assuming that the values $q_j$ are positive.  Consequently, we aim to minimize $F(Q)$ with respect to 
\[
Q=\operatorname{diag}(q_1,\dots,q_n),
\]
under the constraint $|Q|=1$, that is
\beq
g(Q) =\left(\prod_{j=1}^n q_j^2\right) -1 = 0.
\label{eq:g_constraint} 
\eeq 
Since 
\[
\mathcal F(Q) = \sum_{a \in \mathbb{Z}^n}' \int_1^\infty \left( e^{-\pi Q^{-2}[a]t} t^{-1}
+ e^{-\pi Q^2[a]t}  t^{ n/2 - 1} \right) dt
\]
is a smooth function with respect to \(q_1,\dots,q_n\) on \((0,\infty)^n\), and 
\[
\nabla\!\left(\prod_{j=1}^n q_j^2 - 1\right) \neq 0
\]
whenever $q_j>0$,  
any critical point must satisfy the Lagrange multiplier system of equations. If the solution of the Lagrange multiplier system is unique, then that point is the unique constrained global minimum of $\mathcal F(Q)$. We define the Lagrangian
\[
\mathcal{L}(Q,\lambda)=\mathcal F(Q)-\lambda g(Q).
\]
The minimization problem reduces to solving the system of $n+1$ equations
\begin{align*}
&\begin{cases}
\frac{\partial}{\partial q_j}\mathcal{L}(Q,\lambda)
= \frac{\partial}{\partial q_j}\mathcal F(Q) - 2\lambda q_j \prod_{k\neq j} q^2_k = 0, \qquad j=1,\dots,n, \\[5pt]
\frac{\partial}{\partial \lambda}\mathcal{L}(Q,\lambda)
= g(Q)=\prod_{k=1}^n q_k^2 - 1 = 0.
\end{cases}\\[5pt]
\implies&
\begin{cases}
\left(2q_j\prod_{k\neq j} q^2_k\right)^{-1}\frac{\partial}{\partial q_j}\mathcal F(Q)=\lambda,\qquad j=1,\dots,n, \\[5pt]
\prod_{k=1}^n q_k^2=1.
\end{cases}
\end{align*}
We aim to show that $Q=I$ is the unique solution to this system. Since the same multiplier $\lambda$ must arise from differentiation with respect to any coordinate $q_j$, it follows that the first set of $n$ equations must be independent of the index $j$. For convenience, we set
\[
\alpha_j(Q)=2q_j\prod_{k\neq j} q^2_k.
\]
We will prove that the only solutions to the first set of equations are scalar multiples of the identity, i.e.,
\begin{equation}
\alpha_j(Q)^{-1} \frac{\partial}{\partial q_j}\mathcal F(Q)
=
\alpha_i(Q)^{-1} \frac{\partial}{\partial q_i}\mathcal F(Q)
\quad \forall  i,j
\quad \Longleftrightarrow \quad
Q=\mu I,\ \ \mu\in\mathbb{R}^+.
\end{equation}
The claim then follows immediately from the volume constraint.

First, we assume that $Q=\mu I$ and show that this solves the Lagrangian system. We fix $\{i,j\}\subset\{1,\dots,n\}$ and compute:
\begin{align}
    \alpha_j(Q)^{-1} \frac{\partial}{\partial q_j}\mathcal F(Q)
&=\frac 1 2  \mu^{-(2n-1)}\frac{\partial}{\partial q_j} \mathcal F(Q)\nonumber\\
&=\pi\mu^{-(2n-1)}\sum_{a\in\mathbb{Z}^n}' a_j^2\int_1^\infty \left( \frac{1}{\mu^3} e^{-\pi Q^{-2}[a]t}-\mu e^{-\pi Q^2[a]t}t^{n/2}\  \right) dt.
\end{align}

Consider now the permutation matrix $P$ which interchanges the $i$-th and $j$-th coordinates. For $a=(a_1,\dots,a_j,\dots,a_i,\dots,a_n)$ set $b =Pa=(a_1,\dots,a_i,\dots,a_j,\dots,a_n)$. The map $P:\mathbb{Z}^n\rightarrow\mathbb{Z}^n$ is a bijection on $\mathbb{Z}^n$. Since $Q=\mu I$, we have $Q^{\pm2}[a]=Q^{\pm2}[b]$ and $a_i=b_j$. Thus, by a change of variables in the sum:
\begin{align}
\frac{1}{2}\mu^{-(2n-1)}\frac{\partial}{\partial q_j}\mathcal F(Q)
&=\pi\mu^{-(2n-1)}\sum_{a\in\mathbb{Z}^n}'a_j^2\int_1^\infty \left(   \frac{1}{\mu^3}e^{-\pi Q^{-2}[a]t}-\mu e^{-\pi Q^2[a]t}t^{n/2} \ \right) dt\nonumber\\
&=\pi\mu^{-(2n-1)}\sum_{b\in{\mathbb{Z}^n}}' b_j^2\int_1^\infty \left(  \frac{1}{\mu^3}e^{-\pi Q^{-2}[b]t}-\mu e^{-\pi Q^2[b]t}t^{n/2}\ \right)  dt\nonumber\\
&=\pi\mu^{-(2n-1)}\sum_{a\in{\mathbb{Z}^n}}' a_i^2\int_1^\infty \left( \frac{1}{\mu^3}e^{-\pi Q^{-2}[a]t}-\mu e^{-\pi Q^2[a]t}t^{n/2}\ \right) dt\nonumber\\
&=\frac{1}{2}\mu^{-(2n-1)}\frac{\partial}{\partial q_i} \mathcal F(Q). \nonumber
\end{align}
This shows that $Q=\mu I$ solves the Lagrangian system.  Since we know that there exists a minimizer, if we prove that this is the only solution to the Lagrangian system, then this is the unique minimizer, and by the volume constraint $\mu =1$. So, now we shall prove uniqueness of our solution to the Lagrangian system. Using the volume constraint \( \prod_{k=1}^n q^2_k = 1 \), we compute
\[
\frac{1}{\alpha_j(Q) q_j^3}
=\frac{1}{2q_j^4\prod_{k\neq j}q^2_k}
=\frac{1}{2q_j^2\prod_{k=1}^n q^2_k}
=\frac{1}{2q_j^2},
\]
and similarly
\[
\frac{q_j}{\alpha_j(Q)}
=\frac{q_j}{2q_j\prod_{k\neq j}q_k^2}
=\frac{q_j^2}{2}.
\]
The Lagrange multiplier condition on the derivatives takes the form
\[ \alpha_j(Q)^{-1} \frac{\partial}{\partial q_j}\mathcal F(Q) = 
\pi\sum_{a\in\mathbb{Z}^n}'a_j^2\int_1^\infty \left(
\frac{1}{q_j^2}e^{-\pi Q^{-2}[a]t}
-
q_j^2e^{-\pi Q^2[a]t}t^{n/2} \right) 
\ dt
=\lambda.
\]
We now compare the identities corresponding to two distinct fixed indices $i$ and $j$:
\beq 
&\pi\sum_{a\in\mathbb{Z}^n}'a_j^2\int_1^\infty \left(  \frac{1}{q_j^2}e^{-\pi Q^{-2}[a]t} - q_j^2e^{-\pi Q^2[a]t}t^{n/2} \ \right) dt \nn \\ 
& =
\pi\sum_{a\in\mathbb{Z}^n}'a_i^2\int_1^\infty  \left(  \frac{1}{q_i^2}e^{-\pi Q ^{-2}[a]t} - q_i^2e^{-\pi Q^2[a]t}t^{n/2} \ \right) dt.
\label{eq11}
\eeq 
We simplify this slightly to 
\beq & & \sum_{a \in \Z^n} a_j^2 \int_1 ^\infty \left( \frac{1}{q_j^2} e^{-\pi Q^{-2} [a]t} - q_j^2 e^{-\pi Q^2[a]t} t^{n/2} \right) dt \nn \\ &=&   \sum_{a \in \Z^n} a_i ^2 \int_1 ^\infty \left( \frac{1}{q_i^2} e^{-\pi Q^{-2} [a]t} - q_i^2 e^{-\pi Q^2[a]t} t^{n/2} \right) dt. \label{eq:useconstraint} \eeq 
Here we have included $0$ in the sum because of the terms $a_j^2$ and $a_i^2$ which cause the term corresponding to $a=0$ to vanish anyhow. 
We now introduce the one dimensional theta function and its derivative,
\[
\theta_{q_j}(t)=\sum_{a_j\in\mathbb{Z}}e^{-\pi q_j^2a_j^2 t},
\qquad
\theta_{q_j}'(t)
=
-\pi q_j^2\sum_{a_j\in\mathbb{Z}} a_j^2 e^{-\pi a_j^2q_j^2 t}.
\]
We use this to show that the sums over $\mathbb{Z}^n$ can be written as products of one dimensional series. More precisely, we have
\[
\sum_{a\in\mathbb{Z}^n}a_j^2e^{-\pi Q^{-2}[a]t}
=
\sum_{a_j\in\mathbb{Z}}a_j^2e^{-\pi \frac{a_j^2}{q_j^2}t}
\cdot
\sum_{a_i\in\mathbb{Z}}e^{-\pi \frac{a_i^2}{q_i^2}t}
\dots
\sum_{a_n\in\mathbb{Z}}e^{-\pi \frac{a_n^2}{q_n^2}t},
\]
which can be rewritten in terms of the one dimensional theta function as
\[
\sum_{a\in\mathbb{Z}^n}a_j^2e^{-\pi Q^{-2}[a]t} = -\frac{1}{\pi}q_j^2\theta'_{1/q_j}(t)\prod_{k\neq j}\theta_{1/q_k}(t).
\]
 Similarly 
\[ \sum_{a \in \Z^n} a_i ^2  e^{-\pi Q^{-2} [a] t} = - \frac 1 \pi q_i^2 \theta_{1/q_i} '(t) \prod_{k \neq i} \theta_{1/q_k} (t), \] 
\[ \sum_{a \in \Z^n} a_j ^2 e^{-\pi Q^2[a]t} = - \frac{1}{\pi q_j^2} \theta_{q_j} '(t) \prod_{k \neq j} \theta_{q_k} (t), \] 
\[ \sum_{a \in \Z^n} a_i ^2 e^{-\pi Q^2[a]t} = - \frac{1}{\pi q_i^2} \theta_{q_i} '(t) \prod_{k \neq i} \theta_{q_k} (t). \] 
We use this in \eqref{eq:useconstraint} 
\beq \label{eq:almostdone} &- \frac 1 \pi  \int_1 ^\infty  \theta_{1/q_j} '(t) \prod_{k \neq j} \theta_{1/q_k} (t) dt + \frac 1 \pi \int_1 ^\infty \theta_{q_j} '(t) \prod_{k \neq j} \theta_{q_k} (t) t^{n/2} dt  \\  &= - \frac 1 \pi  \int_1 ^\infty  \theta_{1/q_i} '(t) \prod_{k \neq i} \theta_{1/q_k} (t) dt + \frac 1 \pi \int_1 ^\infty  \theta_{q_i} '(t) \prod_{k \neq i} \theta_{q_k} (t) t^{n/2} dt  \nn \\ 
 &\iff \int_1 ^\infty \left[ \left( \frac{\theta_{1/q_i} '(t)}{\theta_{1/q_i} (t)} - \frac{\theta_{1/q_j} '(t)}{\theta_{1/q_j} (t)} \right) \prod_{k=1} ^n \theta_{1/q_k} (t) + \left( \frac{\theta_{q_j} '(t)}{\theta_{q_j} (t)} - \frac{\theta_{q_i} '(t)}{\theta_{q_i} (t)}\right) \prod_{k=1} ^n \theta_{q_k} (t) t^{n/2} \right] dt = 0.  \nn \eeq 
We therefore study the monotonicity properties of the function
\[
q \longmapsto \frac{\theta'_{q}(t)}{\theta_{q}(t)},
\]
for a fixed $t \geq 1$. Recall that the one dimensional theta function can be written as
\[
\theta(t)= \theta_1(t) =\sum_{a\in\mathbb{Z}}e^{-\pi a^2 t}.
\]
By definition, $\theta_q(t)=\theta(q^2 t)$. We define the variable $u=tq^2$, and we can therefore write
\[
\frac{\theta'_q(t)}{\theta_q(t)} = q^2\frac{\theta'(u)}{\theta(u)}=\frac{1}{t}u\frac{\theta'(u)}{\theta(u)}.
\]
According to \cite[Prop. 3.13]{faulhuber}, the function
\[
u\frac{\theta'(u)}{\theta(u)}
\]
is strictly increasing in $u$ for all $u>0$. Since $u$ is a strictly increasing function of $q$, and the composition of strictly increasing functions is still strictly increasing, the function
\[
q \longmapsto \frac{\theta'_q(t)}{\theta_q(t)}
\]
is strictly increasing in $q$ for every fixed $t\geq1$.
Consequently, if $q_j > q_i$, then for all $t \geq 1$ we have 
 \[  \frac{\theta_{1/q_i} '(t)}{\theta_{1/q_i} (t)} - \frac{\theta_{1/q_j} '(t)}{\theta_{1/q_j} (t)} > 0, \quad \frac{\theta_{q_j} '(t)}{\theta_{q_j} (t)} - \frac{\theta_{q_i} '(t)}{\theta_{q_i} (t)} > 0. \] 
 This would make the integrand in the last line of \eqref{eq:almostdone} strictly positive, a contradiction to the fact that it must vanish.  Thus we must have $q_j \leq q_i$.  By possibly switching names, we may in fact conclude that $q_j=q_i$, and this must hold for all pairs of indices $i,j$.  Thus, together with the volume constraint, we must have $q_j=1$ for all $j$.  
\end{proof}


\providecommand{\bysame}{\leavevmode\hbox to3em{\hrulefill}\thinspace}
\providecommand{\MR}{\relax\ifhmode\unskip\space\fi MR }
\providecommand{\MRhref}[2]{%
  \href{http://www.ams.org/mathscinet-getitem?mr=#1}{#2}
}
\providecommand{\href}[2]{#2}

\end{document}